\documentclass{article}
\usepackage{amsmath}
\usepackage{amsfonts}
\usepackage{amsthm}
\usepackage{amssymb}
\usepackage{enumerate}

\newtheorem{thm}{Theorem}[section]
\newtheorem{prop}[thm]{Proposition}
\newtheorem{cor}[thm]{Corollary}
\newtheorem{lem}[thm]{Lemma}
\newtheorem{conj}[thm]{Conjecture}

\theoremstyle{definition}
\newtheorem{defn}[thm]{Definition}
\newtheorem{rmk}[thm]{Remark}

\begin{document}

\title{Clifford Groups of Arbitrary Quadratic Modules over Commutative Rings}

\author{Shaul Zemel}

\maketitle

\section*{Introduction}

Real Clifford algebras were defined, in some small dimensions, in the second half of the 19th century, in attempts to create a formalism for movements in space. With time they played a larger and larger role in the understanding of quadratic vector spaces, though many of the explicit results were given over the real or complex number fields (see, e.g., \cite{[Ma]} or \cite{[P]}). Still, the most common occurrence of Clifford algebras in the classical sense is over fields, with non-degenerate quadratic forms, where they are, for example, instrumental in the analysis of elements of order 2 in Brauer groups of fields of characteristic different from 2 (seem e.g., \cite{[Me]}).

In this classical setting, the invertible elements of a Clifford algebra contain a special subgroup, the Clifford group (sometimes also called the Gspin group), which maps onto the orthogonal group of the quadratic space, with kernel consisting of non-zero scalars. This gives a good description of the structure of this group---in particular it only consists of elements of the Clifford algebra that are homogenous with respect to the natural grading.

There are several generalizations of this notion. One is to allow degenerate forms, as is done in \cite{[Ab]} in some cases. But the more important generalization is to allow the base field to be an arbitrary commutative ring $R$ with 1 (in fact, in this paper we allow both of these generalizations). In the language of the algebraic geometry, this quickly leads to considering quadratic forms over schemes, which need no longer be affine, as is done in \cite{[Au]} and many others. However, we shall be interested in the more straightforward generalization, as is done, for example, in \cite{[Bu]}.

The literature on the subject of Clifford algebras over fields is vast, and so is the literature concerning the scheme-theoretic constructions based on Clifford algebras. However, it seems that not so many papers consider ring-valued quadratic forms on projective modules over a ring in a straightforward manner, like \cite{[Ba]} does for the unimodular case and \cite{[Mc]} for the construction of the Vahlen groups. It is our goal to cover the initial theory in this setting.

\smallskip

In the classical case (a non-degenerate quadratic form over a field), the special algebras obtained by commuting with the ring (namely the center of the Clifford algebra, the center of its even part, the centralizer of the even part, and the ring called the \emph{twisted center} in \cite{[Mc]} and others) are easily determined. The first main result of this paper (Theorem \ref{ZtildeZ} determines these algebras for much more general quadratic forms over rings, under some conditions. Note that there are two submodules of a quadratic module $(M,q)$ measuring degeneracy: $M^{\perp}$ consists of elements that are perpendicular to all of $M$, and $M^{\perp}_{q}$ with the additional requirement of vanishing of the quadratic form, and each plays a different role in the theory (for our determination of these algebras, note that $M^{\perp}$ is contained in the twisted center).

In addition, we consider Clifford groups and orthogonal groups. Note that the latter group is no longer reductive when degeneracy is allowed. Moreover, defining determinants can be more complicated over modules that are not free, and we evaluate them only for some orthogonal maps (the action of many operators on the highest exterior power is not easy to evaluate). Idempotents in the ring play a big role in constructing special elements of the Clifford and orthogonal group (mainly generalizing reflections), and their study is important for proving our results. One important point is that the natural image of the Clifford group is not the full orthogonal group, but rather those orthogonal elements that act trivially on $M^{\perp}$. Our second main result establishes, again under some conditions, an explicit short exact sequence involving the Clifford group and this subgroup of the orthogonal group. We do the same also for the group called, in \cite{[Ma]}, \cite{[Mc]}, and others, the paravector Clifford group.

It is important to note that over rings and with degeneracies, orthogonality of a map does not imply invertibility, and not even injectivity. The triviality on $M^{\perp}$ does imply injectivity, and we conjecture that invertibility also follows. We also pose another conjecture, that is related to the paravector Clifford group mapping only into the special orthogonal group. Both of these conjectures can lead to future research on this subject.

\smallskip

This paper is divided into 3 sections. Section \ref{CAlg} defines Clifford algebras and their subalgebras, and proves the first result involving them. Section \ref{Ortho} considers orthogonal groups and some related constructions (including the determinant). Finally, Section \ref{CGrp} introduces the Clifford group and the paravector Clifford group in this generality, and establishes some of their properties.

\section{Clifford Algebras over Commutative Rings \label{CAlg}}

Let $R$ be a commutative ring with unit, and let $M$ be a finitely generated projective module over $R$. A \emph{quadratic form} on $M$ is a map $q:M \to R$ such that $q(ax)=a^{2}q(x)$ for every $a \in R$ and $x \in M$, and such that the map
\begin{equation}
(\cdot,\cdot):M \times M \to R,\qquad(x,y):=q(x+y)-q(x)-q(y) \label{bilform}
\end{equation}
is $R$-bilinear (and clearly symmetric). The pair $(M,q)$ is called, in short, a \emph{quadratic module}. Elements $x$ and $y$ of $M$ are called \emph{orthogonal} if $(x,y)=0$, and the element $x \in M$ is called \emph{isotropic} if $q(x)=0$, and \emph{anisotropic} otherwise.

The pairing from Equation \eqref{bilform} defines an $R$-homomorphism $\sigma_{q}$ from $M$ to $M^{*}:=\operatorname{Hom}_{R}(M,R)$ by taking $x$ to $y \mapsto (x,y)$ (which, by symmetry, is the same as $y \mapsto (y,x)$ by symmetry). Using it, we define the \emph{kernel} and the \emph{quadratic kernel} of $(M,q)$, defined by
\begin{equation}
M^{\perp}:=\ker\sigma_{q}=\{x \in M|(x,y)=0\ \forall y \in M\}\mathrm{\ \ and\ \ }M_{q}^{\perp}:=\{x \in M^{\perp}|q(x)=0\}   \label{kernel}
\end{equation}
respectively. Note that for $x \in M$ and $y \in M^{\perp}$ we have $q(x+y)=q(x)$, so that by setting $\overline{M}:=M/M^{\perp}$ and $p:M\to\overline{M}$ to be the natural projection, the (finitely generated) module $\overline{M}$ comes with a quadratic form $\overline{q}$ such that $q=\overline{q} \circ p$ as maps from $M$ to $R$. Moreover, the quadratic form $\overline{q}$ on $\overline{M}$ yields the homomorphism $\sigma_{\overline{q}}:\overline{M}\to\overline{M}^{*}:=\operatorname{Hom}_{R}(\overline{M},R)$ also when $\overline{M}$ is not projective. In addition, if we identify $\overline{M}^{*}$ with the set of elements of $M^{*}$ that send $M_{q}^{\perp}$ to 0 (this resulting injection is the dual $p^{*}$ of the projection $p:M\to\overline{M}$), then it is clear that $\sigma_{q}:M \to M^{*}$ is the composition $p^{*}\circ\sigma_{\overline{q}} \circ p$.

We say that the bilinear form from Equation \eqref{bilform}, and with it the quadratic form $q$ and the quadratic module $(M,q)$, is \emph{non-degenerate} if $\sigma_{q}$ is injective, i.e., if $M^{\perp}=\{0\}$ (and \emph{degenerate} otherwise), and \emph{unimodular} (also often \emph{regular}) in case it is an isomorphism. When $R$ is a field $\mathbb{F}$, the latter two notions coincide, which explains why many authors (like \cite{[Ba]} and \cite{[Mc]}) call the quadratic module non-degenerate (or non-singular) only when $\sigma_{q}:M \to M^{*}$ is an isomorphism also over more general rings. However, many results that are stated for unimodular quadratic modules hold more generally for non-degenerate ones, with the same proofs, which motivates our distinction. We call $q$, or $(M,q)$, \emph{quadratically non-degenerate} when $M_{q}^{\perp}=\{0\}$.
\begin{rmk}
On one hand, the fact that $M_{q}^{\perp} \subseteq M^{\perp}$ implies that any non-degenerate quadratic module is quadratically non-degenerate, and it is clear that the kernel of $\sigma_{\overline{q}}$ is $\overline{M}^{\perp}=M^{\perp}/M_{q}^{\perp}$ (and $\overline{M}_{q}^{\perp}=0$). On the other hand, for every $x \in M$ we have $2q(x)=q(2x)-2q(x)=(x,x)$, meaning that if $x \in M^{\perp}$ then $q(x)$ lies in the ideal $R[2]:=\{a \in R|2a=0\}$ of $R$. In particular, if 2 is not a zero-divisor in $R$ then $M_{q}^{\perp}=M^{\perp}$, and the two notions of non-degeneracy coincide. More generally, the latter property yields $2M^{\perp} \subseteq M_{q}^{\perp}$, and combining it with Equations \eqref{bilform} and \eqref{kernel} shows that the restriction of $q$ to $M^{\perp}$ is an additive map into the ideal $R[2]:=\{a \in R|2a=0\}$ of $R$. Thus, considering the quotient $M^{\perp}/2M^{\perp}$ and the ideal $R[2]$ as modules over the $\mathbb{F}_{2}$-algebra $R/2R$, the map $M^{\perp}/2M^{\perp} \to R[2]$ induced from $q$ is Frobenius-linear, and its kernel is $M_{q}^{\perp}/2M^{\perp}$ (so that the restriction of $\overline{q}$ to the kernel $\overline{M}^{\perp}$ of $\sigma_{\overline{q}}$ is a Frobenius-linear injective map into $R[2]$).  \label{kersigmaq}
\end{rmk}

In some cases, like when $R=\mathbb{F}$ is a field of characteristic different from 2, one can find, for a quadratic module $(M,q)$ over $R$, a set $\{x_{i}\}_{i=1}^{n} \subseteq M$ that generate $M$ freely over $R$ and are mutually orthogonal with respect to the bilinear form from Equation \eqref{bilform}. We then say that $\{x_{i}\}_{i=1}^{n}$ are \emph{free orthogonal generators} for $(M,q)$ over $R$. The following lemma uses this notion to exemplify Remark \ref{kersigmaq}.
\begin{lem}
If $M$ contains a set $\{x_{i}\}_{i=1}^{n}$ of free orthogonal generators then $M^{\perp}$ consists of those elements $\sum_{i=1}^{n}a_{i}x_{i} \in M^{\perp}$ for which $2a_{i}q(x_{i})=0$ for every $i$, and such an element lies in $M_{q}^{\perp}$ if and only if $\sum_{i=1}^{n}a_{i}^{2}q(x_{i})=0$. In particular, if $R$ is an $\mathbb{F}_{2}$-algebra then $\sigma_{q}=0$ and $M^{\perp}=M$, while $M_{q}^{\perp}$ is defined by the same equality. \label{fog2}
\end{lem}

\begin{proof}
Elements of $M$ can be uniquely written as $\sum_{i=1}^{n}a_{i}x_{i}$ (by free generation), and since we have $(x_{i},x_{j})=0$ when $i \neq j$ and the equality $(x_{i},x_{i})=2q(x_{i})$ from Remark \ref{kersigmaq}, the pairing of our element with $x_{j}$ is $2a_{j}q(x_{j})$. This yields the asserted description of $M^{\perp}$, and for $M_{q}^{\perp}$ we note that the $q$-value of our element is $\sum_{i=1}^{n}a_{i}^{2}q(x_{i})$. Since when $R$ is an $\mathbb{F}_{2}$-algebra the equality $2a_{i}q(x_{i})=0$ always holds, the second assertion follows as well. This proves the lemma.
\end{proof}
Using Lemma \ref{fog2}, one can strengthen Remark \ref{kersigmaq} as follows. By taking an $\mathbb{F}_{2}$-algebra $R$ and values $q(x_{i}) \in R$ such that the equation $\sum_{i=1}^{n}a_{i}^{2}q(x_{i})=0$ has no solutions in $R^{n}$ (e.g., when $n=1$, $q(x_{1})\neq0$, and $R$ is an integral domain), one constructs quadratically non-degenerate modules that are as degenerate as possible, i.e., with $\sigma_{q}=0$.

Recall that in the definition of a quadratic module $(M,q)$, the module $M$ must be projective and finitely generated. These properties do not extend, in general, to the modules $M^{\perp}$ and $M_{q}^{\perp}$ from Equation \eqref{kernel}, and the second property need not be shared by the quotient $\overline{M}$. However, the defining property of projective modules gives us the following result.
\begin{lem}
If $\overline{M}$ is a projective $R$-module, then $M_{q}^{\perp}$ is also projective and finitely generated, and there is a submodule $\tilde{M}$ of $M$ such that for $\tilde{q}:=q|_{\tilde{M}}$ we have $M=M_{q}^{\perp}\oplus\tilde{M}$ and $(\tilde{M},\tilde{q})$ is isomorphic to $(\overline{M},\overline{q})$ as quadratic modules. \label{proj}
\end{lem}

\begin{proof}
We take $\tilde{M}$ to be the image of any homomorphism $\iota:\overline{M} \to M$ that splits the projection $p:M\to\overline{M}$, i.e., such that $p\circ\iota=\operatorname{Id}_{\overline{M}}$ (such a splitting map exists by projectivity). Then $p|_{\tilde{M}}:\tilde{M}\to\overline{M}$ is an isomorphism, and since $q=\overline{q} \circ p$ and $\tilde{q}=q|_{\tilde{M}}$, the isomorphism $(\tilde{M},\tilde{q})\cong(\overline{M},\overline{q})$ follows. The decomposition $M=M_{q}^{\perp}\oplus\tilde{M}$ is clear, and since $M_{q}^{\perp}$ is isomorphic to a quotient of $M$, it is finitely generated. Finally, there exists an $R$-module $N$ such that $M \oplus N$ is free (since $M$ itself is projective), meaning that $M_{q}^{\perp}\oplus(\tilde{M} \oplus N)$ is free and thus $M_{q}^{\perp}$ is also projective. This proves the lemma.
\end{proof}

\smallskip

To a quadratic module $(M,q)$ one attaches the associative $R$-algebra
\begin{equation}
\mathcal{C}:=\mathcal{C}(M,q):=\bigoplus_{r=0}^{\infty}M^{\otimes r}/I,\qquad I=\langle x \otimes x-q(x)\cdot1|x \in M\rangle \label{Cliffdef}
\end{equation}
called the \emph{Clifford algebra} of $(M,q)$, where $1 \in R=M^{\otimes0}$ by definition. It is generated as an $R$-algebra by a set of generators of $M$ over $R$ (like the tensor algebra $\bigoplus_{r=0}^{\infty}M^{\otimes r}$).

The ring $R$ and the module $M$ are naturally embedded in $\mathcal{C}$, and we shall identify them with their images there throughout. It is clear from Equations \eqref{bilform} and \eqref{Cliffdef} that we have
\begin{equation}
xy+yx=(x,y)\cdot1\mathrm{\ in\ }\mathcal{C}\mathrm{\ for\ every\ }x\mathrm{\ and\ }y\mathrm{\ in\ }M. \label{xy+yx}
\end{equation}
It follows from Equation \eqref{xy+yx} that if $x$ and $y$ are orthogonal in $M$ then they anti-commute in $\mathcal{C}$. It is also clear that when $q=0$ the Clifford algebra $\mathcal{C}(M,0)$ is just the exterior algebra $\bigwedge^{*}M:=\bigoplus_{r=0}^{\infty}\bigwedge^{r}M=\bigoplus_{r=0}^{n}\bigwedge^{r}M$, where $n$ is such that $M$ is generated by $n$ elements. The multiplication for arbitrary $q$ is a deformation of that on $\bigwedge^{*}M$, so that in general $\mathcal{C}(M,q)$ is a finitely generated $R$-module: Indeed, if $M$ is generated by $\{x_{i}\}_{i=1}^{n}$ (not necessarily freely) and $[n]:=\mathbb{N}\cap[1,n]$, then for a subset $I\subseteq[n]$ we set $x_{I}:=\prod_{i \in I}x_{i}$ in increasing order (with $x_{\emptyset}=1$), and the $2^{n}$ products $x_{I}$ for such $I$ generate $\mathcal{C}$ as an $R$-module. Moreover, if the generators $\{x_{i}\}_{i=1}^{n}$ of $M$ are free then $\{x_{I}\}_{I\subseteq[n]}$ are free generators for $\mathcal{C}$ as an $R$-module. Since if $M$ is projective then so is $\bigwedge^{r}M$ for every $r$, and the direct sum in $\bigwedge^{*}M$ is finite for finitely generated $M$, we deduce that $\mathcal{C}(M,q)$ is a projective $R$-module for every quadratic module $(M,q)$ (when $R$ is Noetherian, this property can also be obtained using the equivalence of projectivity and local freeness for finitely generated modules).

\smallskip

The fact that the ideal $I$ from Equation \eqref{Cliffdef} is generated by elements supported only on even degrees yields a $\mathbb{Z}/2\mathbb{Z}$-grading on $\mathcal{C}$, in which the even part $\mathcal{C}_{+}$ (resp. the odd part $\mathcal{C}_{-}$) is the image of the direct sum of the tensors $M^{\otimes r}$ with even $r$ (resp. odd $r$). We thus have
\begin{equation}
\mathcal{C}=\mathcal{C}_{+}\oplus\mathcal{C}_{-},\qquad\mathrm{with}\qquad\mathcal{C}_{\varepsilon}\mathcal{C}_{\delta}\subseteq\mathcal{C}_{\varepsilon\delta}\quad\mathrm{for\ }\varepsilon\mathrm{\ and\ }\delta\mathrm{\ in\ }\{\pm\}, \label{grading}
\end{equation}
and both parts $\mathcal{C}_{\pm}$ are projective $R$-modules like $\mathcal{C}$. Using the grading from Equation \eqref{grading} we can write every $\alpha\in\mathcal{C}$ as
\begin{equation}
\alpha=\alpha_{+}+\alpha_{-}\quad\mathrm{with}\quad\alpha_{\pm}\in\mathcal{C}_{\pm},\qquad\mathrm{and\ we\ set}\quad\alpha'=\alpha_{+}-\alpha_{-}. \label{invgrade}
\end{equation}
The map $\alpha\mapsto\alpha'$ from Equation \eqref{invgrade} is an $R$-algebra involution on $\mathcal{C}$, which can be viewed as arising, via the universal property of $\mathcal{C}$, from the natural embedding of $M$ into $\mathcal{C}$ composed with $-\operatorname{Id}_{M}$. The paper \cite{[Ba]} explains the conditions for $\mathcal{C}$ to be a (graded) Azumaya algebra.

There is also an anti-involution $\alpha\mapsto\alpha^{*}$ on $\mathcal{C}$ which restricts to the identity on $M$ (and $R$), sometimes called \emph{transposition}, and as it commutes with the grading involution, their composition gives another involution $\alpha\mapsto\overline{\alpha}:=(\alpha')^{*}=(\alpha^{*})'$, called the \emph{Clifford involution}. The action of all the involutions on $R$ and $M$ (as subsets of $\mathcal{C}$) is very simple, as follows immediately from the definitions.
\begin{lem}
For $a \in R$ and $x \in M$ we have the equalities $a'=a^{*}=\overline{a}=a$, as well as $x^{*}=x$ and $x'=\overline{x}=-x$. \label{RMinv}
\end{lem}
We shall require another simple observation.
\begin{lem}
The set of elements $\alpha\in\mathcal{C}$ that satisfies $\alpha'=\alpha$ is $\mathcal{C}_{+}\oplus\mathcal{C}_{-}[2]$, where the second summand is $\{\alpha\in\mathcal{C}_{-}|2\alpha=0\}$. In particular, this set is $\mathcal{C}_{+}$ if and only if 2 is not a zero-divisor in $R$. \label{annby2}
\end{lem}

\begin{proof}
The first assertion follows immediately by comparing $\alpha$ with $\alpha'$ from Equation \eqref{invgrade}. For the second one, recall that $\mathcal{C}_{-}$ is a projective $R$-module, hence locally free, and 2 is not a zero-divisor in $R$ if and only if it is not a zero-divisor in all of its localizations. This proves the lemma.
\end{proof}

\smallskip

Viewing $\mathcal{C}$ as a ring, and $\mathcal{C}_{+}$ as a subring, it is natural to consider their centers $Z(\mathcal{C})$ and $Z(\mathcal{C}_{+})$. Since $M$ generates $\mathcal{C}$ over $R$, the former can be described as the set of elements $\alpha\in\mathcal{C}$ such that $\alpha x=x\alpha$ for every $x \in M$. However, following \cite{[Mc]}, we recall the involution from Equation \eqref{invgrade} and define the \emph{twisted center}
\begin{equation}
\widetilde{Z}(\mathcal{C}):=\{\alpha\in\mathcal{C}|\alpha x=x\alpha'\ \forall x \in M\} \label{twistZ}
\end{equation}
of $\mathcal{C}$. These algebras, as well as the centralizer $C_{\mathcal{C}}(\mathcal{C}_{+})$ of $\mathcal{C}_{+}$ in $\mathcal{C}$, have the following simple property.
\begin{prop}
The centers $Z(\mathcal{C})$ and $Z(\mathcal{C}_{+})$, the twisted center $\widetilde{Z}(\mathcal{C})$ from Equation \eqref{twistZ}, and the centralizer $C_{\mathcal{C}}(\mathcal{C}_{+})$ are all graded $R$-subalgebra of $\mathcal{C}$ that are preserved under all the involutions. The centralizer $C_{\mathcal{C}}(\mathcal{C}_{+})$ contains the other three algebras, with $\widetilde{Z}(\mathcal{C}_{+})$ being $C_{\mathcal{C}}(\mathcal{C}_{+})_{+}$, and we have $Z(\mathcal{C})_{+}=\widetilde{Z}(\mathcal{C})_{+}$, while an element $\alpha\in\mathcal{C}_{-}$ lies in $\widetilde{Z}(\mathcal{C})_{-}$ if and only if it satisfies $\alpha\beta=\beta'\alpha$ for every $\beta\in\mathcal{C}$. \label{ZCalg}
\end{prop}

\begin{proof}
Take $x \in M$, $\beta\in\mathcal{C}_{+}$, and $\alpha\in\mathcal{C}$, decomposed as in Equation \eqref{grading}. Consider the equalities $\alpha x=x\alpha$, $\alpha x=x\alpha'$, and $\alpha\beta=\beta\alpha$, and decompose both sides of each according to Equation \eqref{grading}. Then the $\pm$ part yields $\alpha_{\mp}x=x\alpha_{\mp}$, $\alpha_{\mp}x=\mp x\alpha_{\mp}$, and $\alpha_{\pm}\beta=\beta\alpha_{\pm}$, showing that $\alpha$ is in each of the algebras if and only if $\alpha_{+}$ and $\alpha_{-}$ are. Thus all algebras are graded, which means that they are closed under the involution from Equation \eqref{invgrade}. Transposing the first and third defining equations and applying the Clifford algebra to the second one gives $x\alpha^{*}=\alpha^{*}x$, $x\overline{\alpha}=\alpha^{*}x$, and $\beta^{*}\alpha^{*}=\alpha^{*}\beta^{*}$, meaning that these algebras are all closed under transposition, and thus also under the Clifford algebra. The fact that $C_{\mathcal{C}}(\mathcal{C}_{+})$ contains $Z(\mathcal{C})$, that $Z(\mathcal{C}_{+})=C_{\mathcal{C}}(\mathcal{C}_{+})_{+}$, and that $Z(\mathcal{C})_{+}=\widetilde{Z}(\mathcal{C})_{+}$ follows immediately from the definitions. Now, from $\widetilde{Z}(\mathcal{C})$ being closed under the grading involution we deduce that $\alpha xy=x\alpha'y=xy\alpha$ for every $\alpha$ in that algebra and $x$ and $y$ in $M$, yielding the remaining inclusion and also showing that $\alpha\beta=\beta_{+}\alpha+\beta_{-}\alpha'$ in the decomposition of $\beta$ from Equation \eqref{grading} for every $\alpha\in\widetilde{Z}(\mathcal{C})$ and $\beta\in\mathcal{C}$. But the right hand side here also equals $\beta\alpha_{+}+\beta'\alpha_{-}$, meaning that $\alpha\in\widetilde{Z}(\mathcal{C})_{-}$ satisfied the last asserted condition, while applying that condition to $\alpha\in\mathcal{C}_{-}$ and $\beta=x \in M$ yields $\alpha x=-x\alpha=x\alpha'$ and thus $\alpha$ is indeed in $\widetilde{Z}(\mathcal{C})$ as desired. This proves the proposition.
\end{proof}

\smallskip

We aim to determine the algebras from Proposition \ref{ZCalg} in many cases. Recall that if $M$ is free over $R$, with generators $\{x_{i}\}_{i=1}^{n}$, then so is $\mathcal{C}:=\mathcal{C}(M,q)$, with generators $\{x_{I}\}_{I\subseteq[n]}$. Given $I\subseteq[n]$ and $i\in[n]$, we let $I \Delta i$ denote the \emph{symmetric difference} of $I$ and $\{i\}$, which equals $I\cup\{i\}$ when $i \notin I$ and $I\setminus\{i\}$ in case $i \in I$.
\begin{lem}
Let $\{x_{i}\}_{i=1}^{n}$ be free orthogonal generators for $(M,q)$ over $R$, and take some $\alpha:=\sum_{I\subseteq[n]}a_{I}x_{I}\in\mathcal{C}$. Then for $I\subseteq[n]$ and $i\in[n]$ there is a constant $b(I,i) \in R$, independent of $\alpha$, such that $\alpha x_{i}=\sum_{I\subseteq[n]}a_{I}b(I,i)x_{I \Delta i}$ and $x_{i}\alpha=\sum_{I\subseteq[n]}a_{I}b(I,i)(-1)^{|I\setminus\{i\}|}x_{I \Delta i}$. \label{signaniso}
\end{lem}

\begin{proof}
The fact that the $x_{I}$'s freely generate $\mathcal{C}$ as an $R$-module when $\{x_{i}\}_{i=1}^{n}$ are free generators for $M$ over $R$ is clear and yields the unique presentations. Recalling that Equation \eqref{xy+yx} and the orthogonality yield the anti-commutativity of the $x_{i}$'s in $\mathcal{C}$, we obtain the formula for $\alpha x_{i}$ with the constant $b(I,i)$ being a sign in case $i \notin I$ and a sign times $q(x_{i})$ when $i \in I$. This sign corresponds to the parity of the number of elements of $I$ that are strictly larger than $i$. For $x_{i}\alpha$ we obtain the same expression, but now with the sign depending on the number of elements of $I$ that are strictly smaller than $i$. As the difference in signs is $(-1)^{|I\setminus\{i\}|}$, this proves the lemma.
\end{proof}
We can now determine $\widetilde{Z}(\mathcal{C})$ and $Z(\mathcal{C})$ explicitly in the free orthogonal generation case.
\begin{prop}
For $M$, $q$, $\{x_{i}\}_{i=1}^{n}$, $\alpha$, and $a_{I}$ as in Lemma \ref{signaniso}, we have $\alpha\in\widetilde{Z}(\mathcal{C})$ if and only if $2a_{I}q(x_{i})=0$ wherever $i \in I$. The condition for $\alpha$ to be in $Z(\mathcal{C})$ is $\alpha_{+}\in\widetilde{Z}(\mathcal{C})_{+}$ and $2a_{I}=0$ for every $I$ with $n>|I|$ odd. \label{orthbasis}
\end{prop}

\begin{proof}
It is clear that $\alpha\in\widetilde{Z}(\mathcal{C})$ if and only if $\alpha x_{i}=x_{i}\alpha'$, and we express the left hand side as in Lemma \ref{signaniso}. For the right hand side, we note that in $\alpha'$ each $a_{I}$ is multiplies by $(-1)^{|I|}$, so that the multiplier of $x_{I \Delta i}$ in $x_{i}\alpha'$ is $a_{I}b(I,i)$ when $i \notin I$ and $-a_{I}b(I,i)$ in case $i \in I$. Thus our equality holds if and only if $a_{I}b(I,i)=-a_{I}b(I,i)$ wherever $i \in I$, which is equivalent to the asserted condition since $b(I,i)$ is a sign times $q(x_{i})$. This proves the first assertion.

For the second one, the equalities arising from even $|I|$ are the same (as $Z(\mathcal{C})_{+}=\widetilde{Z}(\mathcal{C})_{+}$ by Proposition \ref{ZCalg}), and we have to find $Z(\mathcal{C})_{-}$. For $|I|$ odd the comparison arising from $i \in I$ becomes trivial, but when $i \notin I$ we get $a_{I}b(I,i)=-a_{I}b(I,i)$, which yields the required condition since $b(I,i)=\pm1$. Note that when $n$ is odd and $|I|=n$ there is no index $i \notin I$, so that there is no restriction on $a_{[n]}$ for $\alpha$ to be in $Z(\mathcal{C})$ in this case. This proves the proposition.
\end{proof}

\begin{rmk}
The condition for $\widetilde{Z}(\mathcal{C})$ in Proposition \ref{orthbasis} when $|I|=1$ reduces to the one from Lemma \ref{fog2}, meaning that $M\cap\widetilde{Z}(C)=M^{\perp}$ in this case. In fact, Equations \eqref{xy+yx} and \eqref{twistZ} combine with orthogonality and Equation \eqref{kernel} to show that this holds for every quadratic module $(M,q)$, without the need for free orthogonal generators. The condition $2a_{I}=0$ for $Z(\mathcal{C})$ in that proposition is closely related to Lemma \ref{annby2}. \label{twZCM}
\end{rmk}

We can complement Proposition \ref{orthbasis} by giving explicit expressions also for $Z(\mathcal{C}_{+})$ and for the centralizer $C_{\mathcal{C}}(\mathcal{C}_{+})$. Once again we use $I\Delta\{i,j\}$ to denote the symmetric difference of the two sets.
\begin{prop}
Under the notation and conditions from Proposition \ref{orthbasis}, the element $\alpha$ lies in $C_{\mathcal{C}}(\mathcal{C}_{+})$ if the equality $2a_{I}q(x_{i})=0$ holds for every proper subset $I\subsetneq[n]$ and every $i \in I$. We have $\alpha \in Z(\mathcal{C}_{+})$ if and only if $\alpha$ satisfies the same condition together with $a_{I}=0$ for every $I$ of odd cardinality. \label{commC+}
\end{prop}

\begin{proof}
The subalgebra $\mathcal{C}_{+}$ is generated over $R$ by the products $x_{i}x_{j}$ with $i<j$. Using Lemma \ref{signaniso} we evaluate $\alpha x_{i}x_{j}$ as $\sum_{I\subseteq[n]}a_{I}b(I,i)b(I \Delta i,j)x_{I\Delta\{i,j\}}$, and since $i<j$ the last multiplier in each summand is just $b(I,j)$. Arguing as in the proof of that lemma, and using the fact that $i<j$ to compare the sign arising from $I \Delta j$ and $i$ with that coming from $I$ and $i$, we find that \[x_{i}x_{j}\alpha=\sum_{I\subseteq[n]}a_{I}b(I,i)b(I,j)(-1)^{|I\setminus\{i\}|}(-1)^{|I\setminus\{j\}|}x_{I\Delta\{i,j\}}.\] Thus if $I$ either contain both $i$ and $j$ or contains neither then the coefficient of $x_{I\Delta\{i,j\}}$ in $\alpha x_{i}x_{j}$ and in $x_{i}x_{j}\alpha$ is the same, and for sets $I$ that contain only one of those indices, where these coefficients are additive inverses, we may relax the restriction $i<j$ to $i \neq j$ (indeed, Equation \eqref{xy+yx} and orthogonality yield $x_{j}x_{i}=-x_{i}x_{j}$), and assume that $i \in I$ and $j \notin I$. Since $b(I,j)=\pm1$ and $b(I,i)=\pm q(x_{i})$, the condition for centralizing $\mathcal{C}_{+}$ is that $2a_{I}q(x_{i})=0$ wherever $i \in I$ and there exists an index $j$ not contained in $I$. This clearly translates to the desired condition, and by combining it with the equality $\widetilde{Z}(\mathcal{C}_{+})=C_{\mathcal{C}}(\mathcal{C}_{+})_{+}$ from Proposition \ref{ZCalg}, the last assertion is established as well. This proves the proposition.
\end{proof}

\smallskip

For presenting a neater result, we shall need some additional constructions. Recall that $M$ is embedded in $\mathcal{C}=\mathcal{C}(M,q)$, and we have $x^{2}=q(x)$ for every $x \in M$. As mentioned in \cite{[Mc]}, one can describe $\mathcal{C}$ with the embedding of $M$ as the universal pair of an $R$-algebra $A$ and a map $j:M \to A$ such that $j(x)^{2}=q(x)$ for every $x \in M$, in the sense that for every such pair there is a unique $R$-algebra homomorphism $\mathcal{C} \to A$ whose restriction to $M$ is $j$. Consider thus another quadratic module $(N,Q)$, and assume that there is a homomorphism $\varphi:N \to M$ of $R$-modules such that $q\circ\varphi=Q$. The definition of $\mathcal{C}(N,Q)$ with the embedding from $N$ as a universal object implies the existence of an $R$-algebra homomorphism, still denoted by $\varphi$, from $\mathcal{C}(N,Q)$ to $\mathcal{C}(M,q)$, whose restriction to $N\subseteq\mathcal{C}(N,Q)$ is the initial one. In particular, if $(N,Q)$ is a quadratic submodule of $(M,q)$, namely $N \subseteq M$, both are projective and finitely generated, and $q|_{M}=Q$, then $\mathcal{C}(N,Q)$ maps into $\mathcal{C}(M,q)$.

A useful special case is where the quadratic kernel $M_{q}^{\perp}$ of $(M,q)$, defined in Equation \eqref{kernel}, is finitely generated and projective. This gives a ring homomorphism from $\mathcal{C}(M_{q}^{\perp},0)=\bigwedge^{*}M_{q}^{\perp}$ into $\mathcal{C}(M,q)$. In fact, we obtain a ring homomorphism from $\bigwedge^{*}M_{q}^{\perp}$ into $\mathcal{C}(M,q)$ also when $M_{q}^{\perp}$ is not projective, and the exterior algebra is no longer a Clifford algebra in our setting. Also in this case the natural $\mathbb{Z}$-grading on $\bigwedge^{*}M_{q}^{\perp}$ induces a $\mathbb{Z}/2\mathbb{Z}$-grading, using which we write $\bigwedge^{*}M_{q}^{\perp}$ as a direct sum of $(\bigwedge^{*}M_{q}^{\perp})_{+}$ and $(\bigwedge^{*}M_{q}^{\perp})_{-}$ like in Equation \eqref{grading}, and the map into $\mathcal{C}(M,q)$ still respects the grading.

When the quotient module $\overline{M}$ is also projective, and the Clifford algebra $\mathcal{C}(\overline{M},\overline{q})$ is defined, the map $p:M\to\overline{M}$ extends, by the same argument, to a (surjective) map $p:\mathcal{C}(M,q)\to\mathcal{C}(\overline{M},\overline{q})$. In addition, as in the proof of Lemma \ref{proj}, there is a homomorphism $\iota:\overline{M} \to M$ splitting $p$. Since it is clear that $q\circ\iota=\overline{q}$, the (injective) homomorphism $\iota$ extends to an (injective) $R$-algebra homomorphism $\iota:\mathcal{C}(\overline{M},\overline{q})\to\mathcal{C}(M,q)$, which still satisfies $p\circ\iota=\operatorname{Id}_{\mathcal{C}(\overline{M},\overline{q})}$. We have seen in that lemma that $M_{q}^{\perp}$ is also projective, and if we write the Clifford algebra $\mathcal{C}(M_{q}^{\perp},0)=\bigwedge^{*}M_{q}^{\perp}$ as $R\oplus\bigoplus_{r>0}\bigwedge^{r}M_{q}^{\perp}$, then we recall, Following \cite{[Ab]}, that all the elements in the second summand are nilpotent there. Moreover, the ideal in $\mathcal{C}(M,q)$ generated by the images of all these elements is nilpotent in $\mathcal{C}(M,q)$, and we obtain the following result.
\begin{prop}
Assume that $\overline{M}$ is a projective $R$-module, and that the quadratic module $(\overline{M},\overline{q})$ is unimodular. Then $p:\mathcal{C}(M,q)\to\mathcal{C}(\overline{M},\overline{q})$ is the quotient map of the former algebra modulo its radical. \label{rad}
\end{prop}

\begin{proof}
The unimodularity of the quadratic module $(\overline{M},\overline{q})$ implies, via the theorem in Subsection (2.3) of \cite{[Ba]}, that $\mathcal{C}(\overline{M},\overline{q})$ is an Azumaya algebra, hence has no radical. As the kernel of the map from $\mathcal{C}(M,q)$ is generated by $\bigoplus_{r>0}\bigwedge^{r}M_{q}^{\perp}$, it is indeed the radical of that algebra. This proves the proposition.
\end{proof}
It will be interesting to see if $\mathcal{C}(\overline{M},\overline{q})$ is an Azumaya algebra also when the unimodularity condition is relaxed somewhat. For example, one may ask whether is this the case when $R$ is an integral domain with $2\neq0$ and $(\overline{M},\overline{q})$ is non-degenerate in our sense. This will extend the validity of Proposition \ref{rad} to much more general situations. However, the proof from \cite{[Ba]} uses an isomorphism with a hyperbolic module, which is not a natural one when $(\overline{M},\overline{q})$ is not unimodular.

\smallskip

Let $(M,q)$ be a quadratic $R$-module, and take $S$ to be any $R$-algebra. Then $M_{S}:=M\otimes_{R}S$ is a finitely generated projective module over $S$, and the map $q_{S}:M_{S} \to S$ given by \[\textstyle{q_{S}\big(\sum_{i=1}^{l}a_{i} \otimes x_{i}\big):=\sum_{i=1}^{l}a_{i}^{2}q(x_{i})+\sum_{1 \leq i<j \leq l}a_{i}a_{j}(x_{i},x_{j})}\] is a well-defined quadratic form on $M_{S}$, with the associated bilinear form from Equation \eqref{bilform} being \[\textstyle{\big(\sum_{i=1}^{l}a_{i} \otimes x_{i},\sum_{j=1}^{k}b_{j} \otimes y_{j}\big)_{S}=\sum_{i=1}^{l}\sum_{j=1}^{k}a_{i}b_{j}(x_{i},y_{j})}.\] It is clear that the Clifford algebra $\mathcal{C}(M_{S},q_{S})$ over $S$ is the same as $\mathcal{C}(M,q)\otimes_{R}S$, and the resulting natural $R$-algebra homomorphism from $\mathcal{C}(M,q)$ to $\mathcal{C}(M_{S},q_{S})$ can be viewed as obtained by the homomorphism $i:M \to M_{S}$ of $R$-modules using the universal property, since it is clear that $q_{S} \circ i=q$.

\smallskip

Using this notation, we can now prove the following theorem, which generalizes the classical results for quadratic spaces over fields of characteristic different from 2 given in, e.g., \cite{[Ba]} for the non-degenerate case (see \cite{[Ab]} for some results in the degenerate case over $\mathbb{R}$ and $\mathbb{C}$). Recall that when $(M,q)$ is a quadratic $R$-module, $M$ must be a projective $R$-module, and thus so are its wedge powers, implying that $\mathcal{C}(M,q)$ is projective (thus flat) as an $R$-module. We also define a \emph{projective quadratic $R$-algebra} to be an $R$-algebra which is given by an external direct sum $R \oplus P\beta$, where $P$ is a projective module of rank 1 over $R$ and $\beta$ is an element such that $a^{2}\beta^{2} \in R$ for every $a \in P$.
\begin{thm}
Let $(M,q)$ be a quadratic module over $R$, and set $\mathcal{C}:=\mathcal{C}(M,q)$. Assume that $R$ is contained in an algebra $S$ in which 2 is not a zero-divisor, such that $(M_{S},q_{S})$ has a free orthogonal generating set $\{x_{i}\}_{i=1}^{n}$, such that the values $q_{S}(x_{i})$ either vanish or are not zero-divisors in $S$. Then the twisted center $\widetilde{Z}(\mathcal{C})$ from Equation \eqref{twistZ} is the image of $\bigwedge^{*}M^{\perp}$ in $\mathcal{C}$. If $q=0$ then $C_{\mathcal{C}}(\mathcal{C}_{+})$ is also this exterior algebra, and $Z(\mathcal{C})$ and $Z(\mathcal{C}_{+})$ both equal the even part $(\bigwedge^{*}M^{\perp})_{+}$ of $\bigwedge^{*}M^{\perp}$. If $q\neq0$ then $\mathcal{C}$ contains a projective quadratic $R$-algebra $R \oplus P\beta$, with $\beta^{2}$ vanishing if and only if $q$ is degenerate, such that $C_{\mathcal{C}}(\mathcal{C}_{+})=\widetilde{Z}(\mathcal{C}) \oplus P\beta$. Finally, if the rank $n$ of $M$ is even then $Z(\mathcal{C})$ is the even part $(\bigwedge^{*}M^{\perp})_{+}$ of $\bigwedge^{*}M^{\perp}$ and $Z(\mathcal{C}_{+})$ is $Z(\mathcal{C}) \oplus P\beta$, but when $n$ is odd, $Z(\mathcal{C}_{+})=(\bigwedge^{*}M^{\perp})_{+}$ and $Z(\mathcal{C})=(\bigwedge^{*}M^{\perp})_{+} \oplus P\beta$. \label{ZtildeZ}
\end{thm}

\begin{proof}
Observe that 2 is not a zero-divisor in $R$ under our assumption, so that Remark \ref{kersigmaq} gives $M_{q}^{\perp}=M^{\perp}$ (and non-degeneracy is the same as quadratic non-degeneracy) and we can omit the index $q$ in $\bigwedge^{*}M_{q}^{\perp}$ and its $\pm$ parts.

Assume first that $S=R$, i.e., that $\{x_{i}\}_{i=1}^{n}$ is a set of free orthogonal generators of $(M,q)$ over $R$. Then we apply Propositions \ref{orthbasis} and \ref{commC+}, and our assumptions on 2 and $q(x_{i})$ in $R=S$ replaces all the conditions of the sort $2a_{I}=0$ or $2a_{I}q(x_{i})=0$ by simply $a_{I}=0$. These propositions, which also imply that the element $x_{[n]}$ is central in $\mathcal{C}$ when $n$ is odd and lies in $\mathcal{C}_{+}$ when $n$ is even, establish the asserted results in this case (the projective quadratic $R$-algebra in this case is $R \oplus Rx_{[n]}$, which is a free $R$-module).

In the general case, what we have proved determines the corresponding subalgebras of $\mathcal{C}_{S}:=\mathcal{C}(M_{S},q_{S})$ (as a Clifford algebra over $S$). Moreover, the fact that $\mathcal{C}$ is flat over $R$ and $R$ is embedded in $S$ implies that $\mathcal{C}$ is embedded in $\mathcal{C}_{S}$, so that equalities from $\mathcal{C}$ hold in $\mathcal{C}_{S}$ only if they were valid in $\mathcal{C}$.

Now, centrality is equivalent to commuting with $M\subseteq\mathcal{C}$, and twisted centrality is also defined in terms of multiplication by $M$ in Equation \eqref{twistZ}. It follows, by extending scalars from $M$ to $M_{S}$, that our subalgebras of $\mathcal{C}$ are the intersections of those of $\mathcal{C}_{S}$ with $\mathcal{C}$. Note that the same argument identifies the kernel $(M_{S})^{\perp}$ appearing in the subalgebras of $\mathcal{C}_{S}$ with the image of $(M^{\perp})_{S}=M^{\perp}\otimes_{R}S$ inside $M_{S}$, and we write this kernel just as $M_{S}^{\perp}$. Moreover, for the conditions separating the cases we observe that $q$ and $q_{S}$ vanish and are degenerate together, and the rank of $M$ over $R$ is the same rank $n$ of $M_{S}$ over $S$.

We thus have to determine the intersection of $\bigwedge^{*}M_{S}^{\perp}$ and of its even part $(\bigwedge^{*}M_{S}^{\perp})_{+}$ with $\mathcal{C}$, and when $q\neq0$ we also need to intersect the projective quadratic $S$-algebra $S \oplus Sx_{[n]}$ with $\mathcal{C}$. The first two intersections clearly give $\bigwedge^{*}M^{\perp}$ and $(\bigwedge^{*}M^{\perp})_{+}$ respectively, since $M_{S}^{\perp} \cap M=M^{\perp}$ when $R$ is embedded in $S$. As for the third algebra, of which we only require the part $Sx_{[n]}$, recall that this part is $\bigwedge^{n}_{S}M_{S}$, in which $\bigwedge^{n}M$ embeds by flatness. As the latter is a projective $R$-module of rank 1, it is of the form $Px_{[n]}$ for such an $R$-module $P \subseteq S$. Since $x_{[n]}^{2} \in S$ and $Px_{[n]}\subseteq\mathcal{C}$, we obtain the condition $a^{2}x_{[n]}^{2} \in R$ for every $a \in P$, required for $R \oplus Px_{[n]}$ to indeed be a projective quadratic $R$-algebra as desired. This proves the theorem.
\end{proof}

Theorem \ref{ZtildeZ} gives the well-known results from \cite{[Ba]} in case $R$ is a field $\mathbb{F}$ of characteristic different from 2 (and $S=\mathbb{F}$ as well) and $q$ is non-degenerate, and extends them directly to the degenerate setting. Another case for which Theorem \ref{ZtildeZ} is easily applied is where $R$ is an integral domain, $S$ is its fraction field $\mathbb{F}$, and the characteristic is different from 2. But its result is valid in many other situations.

\begin{rmk}
In contrast to Theorem \ref{ZtildeZ}, however, when $R$ is an $\mathbb{F}_{2}$-algebra all the equalities from Propositions \ref{orthbasis} and \ref{commC+} hold for every element of $\mathcal{C}$, so that we have $Z(\mathcal{C})=\widetilde{Z}(\mathcal{C})=C_{\mathcal{C}}(\mathcal{C}_{+})=\mathcal{C}$ and $Z(\mathcal{C}_{+})=\mathcal{C}_{+}$. Indeed, in this case the involution from Equation \eqref{invgrade} is trivial (and $\mathcal{C}_{-}[2]$ from Lemma \ref{annby2} is all of $\mathcal{C}_{-}$), so that $Z(\mathcal{C})=\widetilde{Z}(\mathcal{C})$. Moreover, the existence of free orthogonal generators implies that $\sigma_{q}=0$ (see Remark \ref{kersigmaq}), and thus Equation \eqref{xy+yx} reduces to every two elements of $M$ commuting in $\mathcal{C}$. This makes $\mathcal{C}$ the (commutative) compositum of the quadratic $R$-algebras $R[x_{i}]/\big(x_{i}^{2}-q(x_{i})\big)$ (with its $\mathbb{Z}/2\mathbb{Z}$-grading), and all the centers and centralizers are the full algebras. \label{char2}
\end{rmk}

\section{Orthogonal Groups \label{Ortho}}

We recall that a map $\varphi:M \to M$ is called \emph{orthogonal}, or an \emph{isometry}, in case $q\circ\varphi=q$. It is clear that an isometry of $(M,q)$ must preserve the submodules $M^{\perp}$ and $M_{q}^{\perp}$ of $M$ defined in Equation \eqref{kernel}, and by taking the quotient one obtains an isometry of $(\overline{M},\overline{q})$ (recall that the latter is a quadratic module only when $\overline{M}$ is projective in our definition, but the meaning of an isometry is clear also without this assumption). Note that an orthogonal transformation need not be injective in general: Indeed, when $q=0$ every endomorphism of $M$ (including 0) is orthogonal, and more generally so is the composition $M\to\overline{M}\cong\tilde{M} \hookrightarrow M$, with $\tilde{M}$ from Lemma \ref{proj}, when $\overline{M}$ is projective. However, it is clear that the kernel of an orthogonal transformation must be contained in $M_{q}^{\perp}$, thus in $M^{\perp}$.

Let $\operatorname{O}(M,q)$ be the group of invertible orthogonal transformations of $(M,q)$. It is clear that if $\varphi$ is orthogonal and invertible then $\varphi^{-1}$ is also orthogonal, thus $\operatorname{O}(M,q)$ is indeed a group for every quadratic module $(M,q)$. We extend the notation $\operatorname{O}(\overline{M},\overline{q})$ also to the case where $\overline{M}$ is not necessarily projective, and for an $R$-module $M_{q}^{\perp}$ we denote by $\operatorname{GL}(M_{q}^{\perp})$ the group of invertible maps in $\operatorname{End}_{R}(M_{q}^{\perp})$. One then has the following exact sequence of groups.
\begin{lem}
Using this notation, the group $\operatorname{O}(M,q)$ lies in the exact sequence \[0\to\big(\operatorname{Hom}_{R}(\overline{M},M_{q}^{\perp}),+\big)\to\operatorname{O}(M,q)\to\operatorname{GL}(M_{q}^{\perp})\times\operatorname{O}(\overline{M},\overline{q}).\]
In case the sequence $0 \to M_{q}^{\perp} \to M\stackrel{p}{\to}\overline{M}\to0$ splits, the asserted sequence is exact with an additional 0 in the end, and splits as well. \label{Odecom}
\end{lem}

\begin{proof}
We have already seen that if $\varphi\in\operatorname{O}(M,q)$ then $\varphi$ preserves $M_{q}^{\perp}$, and since the same applies to $\varphi^{-1}$, we find that $\varphi|_{M_{q}^{\perp}}\in\operatorname{GL}(M_{q}^{\perp})$. It follows that $\varphi$ also operates bijectively and orthogonally on the quotient module $(\overline{M},\overline{q})$, and as both maps are clearly homomorphisms of groups, this defines the rightmost group homomorphism. Now, on one hand, given $\varphi$ in the kernel of this map, the map $\tilde{\psi}:M \to M$, $\tilde{\psi}:x\mapsto\varphi(x)-x$ takes values in $M_{q}^{\perp}$ (since it becomes trivial modulo $M_{q}^{\perp}$ by inducing $\operatorname{Id}_{\overline{M}}$ on the quotient) and is trivial on $M_{q}^{\perp}$ (since $\varphi|_{M_{q}^{\perp}}=\operatorname{Id}_{M_{q}^{\perp}}$), so that it is induced from a map $\psi:\overline{M} \to M_{q}^{\perp}$. On the other hand, for every $\psi\in\operatorname{Hom}_{R}(\overline{M},M_{q}^{\perp})$, with induced map $\tilde{\psi}:M \to M_{q}^{\perp} \subseteq M$, it is clear that $\tilde{\psi}^{2}(x)=0$ for every $x \in M$, so that the map $\varphi:x \to x+\tilde{\psi}(x)$ is bijective (with inverse $x \mapsto x-\tilde{\psi}(x)$) and orthogonal by the definition of $M_{q}^{\perp}$ in Equation \eqref{kernel}. This determines the kernel as $\operatorname{Hom}_{R}(\overline{M},M_{q}^{\perp})$, and the group structure is clearly the additive one.

Now, the sequence of modules splits if and only if there is a submodule $\tilde{M}\cong\overline{M}$ of $M$ such that $M=M_{q}^{\perp}\oplus\tilde{M}$, and set $\tilde{q}=q|_{\tilde{M}}$. A choice of such a decomposition embeds $\operatorname{GL}(M_{q}^{\perp})$ and $O(\tilde{M},\tilde{q})\cong\operatorname{O}(\overline{M},\overline{q})$ into $\operatorname{O}(M,q)$ by acting on each summand separately, which gives the surjectivity of the map to the product and the splitting of the sequence. This proves the lemma.
\end{proof}
We recall from Lemma \ref{proj} that the splitting scenario occurs wherever $\overline{M}$ is projective. In this case we write $M$ as $M_{q}^{\perp}\oplus\tilde{M}$, with $M_{q}^{\perp}$ finitely generated and projective, and thus, at least when $R$ is Noetherian, the dual $(M_{q}^{\perp})^{*}$ is also finitely generated and projective (indeed, $M_{q}^{\perp}$ is locally free of finite rank, yielding the same properties for $(M_{q}^{\perp})^{*}$). We can embed $M$ in $(M_{q}^{\perp})^{*} \oplus M$, which is the orthogonal direct sum of the hyperbolic module associated with $M_{q}^{\perp}$ (which is unimodular since $M_{q}^{\perp}$ is locally free and thus reflexive) and the (quadratically non-degenerate) quadratic module $(\tilde{M},\tilde{q})\cong(\overline{M},\overline{q})$. Then, inside the orthogonal group of this larger quadratic module, the structure of the subgroup that preserves $M_{q}^{\perp}$ can be described, in some cases, using the results of, e.g., \cite{[Z]} (note that while that reference considers only the case $R=\mathbb{Z}$, the fact that here the quadratic module splits nicely over $R$ and the chosen complement $(M_{q}^{\perp})^{*}$ is isotropic takes the structure theorem appearing in Theorem 3.2 of that reference into the simpler form existing over fields, as in Lemmas 1.1 and 1.4 and Proposition 1.6 there, which makes it easy to extend these results to a much more general setting). Our Lemma \ref{Odecom} is thus in correspondence with the restriction to $M=\big\{x\in(M_{q}^{\perp})^{*} \oplus M|(x,y)=0\ \forall y \in M_{q}^{\perp}\}$ amounting to replacing the Heisenberg group from that reference by its quotient $\operatorname{Hom}_{R}(\overline{M},M_{q}^{\perp})$.

\smallskip

The subgroup of $\operatorname{O}(M,q)$ consisting of those elements acting trivially on $M^{\perp}$ is just the kernel of the kernel the composition of the map from Lemma \ref{Odecom} with the projection onto the $\operatorname{GL}(M_{q}^{\perp})$ factor. However, a slightly smaller subgroup that will be related to Clifford groups later is the subgroup $\operatorname{O}_{M^{\perp}}(M,q)$, consisting of those elements of $\operatorname{O}(M,q)$ whose action on the (possibly larger) submodule $M^{\perp}$ is trivial. It is contained in the aforementioned kernel, a containment that can be strict. The analogue of Lemma \ref{Odecom} for that group is as follows.
\begin{lem}
Set $\overline{M}_{\sigma}:=M/M^{\perp}=\overline{M}/\overline{M}^{\perp}$. Then we have an exact sequence \[0\to\big(\operatorname{Hom}_{R}(\overline{M}_{\sigma},M_{q}^{\perp}),+\big)\to\operatorname{O}_{M^{\perp}}(M,q)\to\operatorname{O}(\overline{M},\overline{q}),\] that is also a splitting short exact sequence in case $p:M\to\overline{M}$ splits. \label{OMperp}
\end{lem}

\begin{proof}
Any element of $\operatorname{O}(\overline{M},\overline{q})$ preserves $\overline{M}^{\perp}$, and since $\overline{q}$ is injective on $\overline{M}^{\perp}$ (see Remark \ref{kersigmaq}), restricting elements of $\operatorname{O}(\overline{M},\overline{q})$ to $\overline{M}^{\perp}$ gives $\operatorname{Id}_{\overline{M}^{\perp}}$. It follows that the natural image of the restriction of the map from Lemma \ref{Odecom} (without the factor $\operatorname{GL}(M_{q}^{\perp})$) to $\operatorname{O}_{M^{\perp}}(M,q)$ is the full group $\operatorname{O}(\overline{M},\overline{q})$. Next, as the element $\varphi\in\operatorname{O}(M,q)$ that is associated with $\psi\in\operatorname{Hom}_{R}(\overline{M},M_{q}^{\perp})$ in the proof of Lemma \ref{Odecom} is $\operatorname{Id}_{M}+\tilde{\psi}$ for the map $\tilde{\psi}:M \to M_{q}^{\perp}$ associated with $\psi$, we deduce that $\varphi\in\operatorname{O}_{M^{\perp}}(M,q)$ if and only if $\tilde{\psi}|_{M^{\perp}}=0$ if and only if $\psi|_{\overline{M}^{\perp}}=0$. This happens precisely when $\tilde{\psi}$ (or $\psi$) arises from a map from the possibly coarser quotient $\overline{M}_{\sigma}$, establishing the asserted sequence. In the case where $p$ splits and $M=M_{q}^{\perp}\oplus\tilde{M}$, we have $M^{\perp}=M_{q}^{\perp}\oplus\tilde{M}^{\perp}$ with $\tilde{M}^{\perp}\cong\overline{M}^{\perp}$, and as elements of $\operatorname{O}(\overline{M},\overline{q})$ were seen to act trivially on $\overline{M}^{\perp}$, lifting such an element to operate on $\tilde{M}$ with $\operatorname{Id}_{M_{q}^{\perp}}$ indeed gives an element of $\operatorname{O}_{M^{\perp}}(M,q)$. This proves the lemma.
\end{proof}

\smallskip

Since the kernel of an orthogonal map lies in $M_{q}^{\perp}$, it is clear that triviality of the action on that module (and in particular triviality on $M^{\perp}$) implies injectivity. Note that injectivity and orthogonality do not imply surjectivity in general: Indeed, take any non-zero projective $\mathbb{Z}$-module $M$, with $q=0$, and the existence of injective but not surjective maps from $M$ to $M$ exemplifies this. On the other hand, if $R$ is any field (also in characteristic 2), then injectivity implies surjectivity in finite-dimensional vector spaces, and every orthogonal map acting trivially on $M^{\perp}$ lies in $\operatorname{O}_{M^{\perp}}(M,q)$. As another example, if $M$ is a free module $M$ such that the determinant of a Gram matrix of the bilinear form is not a zero-divisor (so in particular $q$ is non-degenerate). In this case any orthogonal transformation is represented by a matrix, the usual argument implies that the determinant of this matrix squares to 1 (the condition on the determinant of the Gram matrix allows us to delete it from the equation), and thus the matrix, and with it the transformation, is invertible. In addition, if $R$ is an integral domain then one can extend scalars to its field of fractions $\mathbb{F}$, and $M_{\mathbb{F}}$ splits into the direct sum of $M^{\perp}_{\mathbb{F}}$ and some non-degenerate complement that is isomorphic to the quotient (this quotient need not be $\overline{M}_{\mathbb{F}}$ in characteristic 2, as Remark \ref{kersigmaq} and Lemma \ref{fog2} show). The non-degeneracy of the quadratic form on the latter and the $\operatorname{Id}_{M^{\perp}}$ assumption show that the determinant of such a transformation is $\pm1$, making its inverse defined over $R$, thus preserving $M$, and yielding the invertibility of the transformation on $M$.

Based on this evidence, we pose the following conjecture.
\begin{conj}
Given any commutative ring $R$, and any quadratic module $(M,q)$ over $R$, every orthogonal map from $M$ to itself that restricts to $\operatorname{Id}_{M^{\perp}}$ on $M^{\perp}$ is bijective, hence lies in $\operatorname{O}_{M^{\perp}}(M,q)$. \label{allMq}
\end{conj}
As bijectivity is local, and projective modules become free over local rings, it suffices for Conjecture \ref{allMq} to consider matrices over local rings that preserve a given Gram matrix, with some identity condition. Note, however, that if $R$ has zero-divisors then it may happen that $x \in M$ will not be in $M^{\perp}$ but there will be $a \in R$ with $0 \neq ax \in M^{\perp}$, making this analysis more delicate. We leave this question to further research.

\smallskip

Subsection 4.1 of \cite{[Ba]}, as well as \cite{[Mc]} and other references, consider Euler transformations and reflections as orthogonal maps, which over fields have determinant 1 and $-1$ respectively. Over an arbitrary ring $R$ these notion can be generalized, and also the determinant, once defined, can attain more values in $R$, like elements that square to 1 in $R$ that need not necessarily by $\pm1$. The set that is denoted by $\mathbb{Z}_{2}(R)$ can be identified with the set $\operatorname{Idem}(R)$ of \emph{idempotents} in $R$, namely elements $e \in R$ satisfying $e^{2}=e$, and as usual we denote by $\mu_{2}(R)$ the set of $\delta \in R$ with $\delta^{2}=1$.
\begin{lem}
The set $\operatorname{Idem}(R)$, endowed with the multiplication from $R$ and the addition $\oplus$ defined by $e\oplus\epsilon:=e+\epsilon-2e\epsilon$ is an (associative) $\mathbb{F}_{2}$-algebra with unit 1, with $e\oplus1=1-e\in\operatorname{Idem}(R)$ for every $e\in\operatorname{Idem}(R)$. We have the equalities $e\oplus\epsilon=e(1-\epsilon)+(1-e)\epsilon$ and $1-(e\oplus\epsilon)=e\epsilon+(1-e)(1-\epsilon)$. The map $e\mapsto1-2e$ is a group homomorphism from $\big(\operatorname{Idem}(R),\oplus\big)$ to $\mu_{2}(R)$, and its kernel is precisely the intersection $\operatorname{Idem}(R) \cap R[2]$ consisting of those idempotents $e$ of $R$ for which $eR$ is an $\mathbb{F}_{2}$-algebra. \label{idem}
\end{lem}
Note that for Lemma \ref{idem} we consider $\{0\}$ as a trivial $\mathbb{F}_{2}$-algebra, since the trivial idempotent $e=0$ is clearly always in the kernel of the map in question.
\begin{proof}
The first assertion is simple, standard, and straightforward, the second one is clear, and for the third one we easily compare $(1-2e)(1-2\epsilon)$ and $1-2(e\oplus\epsilon)$. Since $1-2e=1$ if and only if $e \in R[2]$, the third assertion is also clear. This proves the lemma.
\end{proof}
The ring structure from Lemma \ref{idem} is also that of a Boolean ring, and the operations are the pointwise operations when idempotents are considered as $\mathbb{F}_{2}$-valued functions on the set of connected components of the space $\operatorname{Spec}(R)$ (i.e., continuouns functions from $\operatorname{Spec}(R)$ to $\mathbb{F}_{2}$). The map into $\mu_{2}(R)$ can be viewed, on connected components, as $e\mapsto(-1)^{e}$ like in Equation (4.4.2) of \cite{[Ba]}, but the restriction of $-1$ to any union $\operatorname{Spec}(eR)$ of connected components should be interpreted as the image of $-1$ in $eR$, which when $eR$ is an $\mathbb{F}_{2}$-algebra becomes the same as 1. As for the surjectivity of the map from Lemma \ref{idem}, it is clear that $\delta\in\mu_{2}(R)$ lies in the image only if the difference $1-\delta$ is a multiple of 2 in $R$. This is not always the case, as one sees by taking $R=\mathbb{Z}[X]/(X^{2}-1)$ and $\delta=X$.

\smallskip

Consider two elements $u$ and $x$ in $M$, with $(u,x)=0$, and such that $q(u)$ annihilates $q(x)$ in $R$ as well as $\sigma_{q}(x)$ in $M^{*}$, i.e., we have $q(u)q(x)=0$ and $q(u)(x,y)=0$ for every $y \in M$ (the latter condition happens for $q(u)=0$, but also when there is $e\in\operatorname{Idem}(R)$ such that $x \in eM$ and $eq(u)=0$, for example). Following Subsection 4.1 of \cite{[Ba]}, we define the map
\begin{equation}
E_{u,x}:M \to M,\qquad E_{u,x}(y):=y+(x,y)u-(u,y)[x+q(x) \cdot u], \label{Euler}
\end{equation}
called the \emph{Euler transformation} associated with $u$ and $x$. It is clearly well-defined and equals $\operatorname{Id}_{M}$ in case either $x$ or $u$ vanishes. These maps have the following properties.
\begin{prop}
The Euler transformation $E_{u,x}$ from Equation \eqref{Euler} lies in $\operatorname{O}_{M^{\perp}}(M,q)$ for every such $u$ and $x$, and if $\overline{u}$ and $\overline{x}$ are the images of $u$ and $x$ in $\overline{M}$ respectively, then the image of the Euler transformation $E_{u,x}$ under the map from Lemma \ref{OMperp} is $E_{\overline{u},\overline{x}}\in\operatorname{O}(\overline{M},\overline{q})$. For any $a \in R$ we have the equality $E_{au,x}=E_{u,ax}$ of Euler transformations, and if $z \in M$ satisfies the same conditions with $u$ then so does the sum $x+z$ and we have $E_{u,x} \circ E_{u,z}=E_{u,x+z}$. Finally, if $v \in M$ satisfies $(u,v)=(v,x)=q(v)q(x)=0$ and $q(v)\sigma_{q}(x)=0$ then $E_{v,x}$, $E_{u+v,x}$, and $E_{q(x)v,u}$ also define Euler transformations, which together satisfy the equality $E_{u,x} \circ E_{v,x}=E_{q(x)v,u} \circ E_{u+v,x}$. \label{propEuler}
\end{prop}

\begin{proof}
The assumption that $(u,x)=q(u)q(x)=q(u)(x,y)=0$ evaluates the $q$-image of the right hand side of Equation \eqref{Euler} is \[q(y)+(x,y)(y,u)+(u,y)^{2}q(x)-(u,y)(y,x)-(u,y)q(x)(y,u)=q(y).\] This implies the orthogonality of $E_{u,x}$, and since the right hand side of Equation \eqref{Euler} reduces to $y$ in case $y \in M^{\perp}$, $E_{u,x}$ restricts to $\operatorname{Id}_{M^{\perp}}$ on $M^{\perp}$. The second one follows immediately by considering the images of both sides of Equation \eqref{Euler} in $\overline{M}$. For the third one, the conditions for $au$ and $x$ as well as for $u$ and $ax$ are evident, and Equation \eqref{Euler} evaluates both $E_{au,x}(y)$ and $E_{u,ax}(y)$ as $y+a(x,y)u-a(u,y)[x+aq(x) \cdot u]$. Next, from the vanishing of $(u,z)$, $q(u)(z,y)$, and $q(u)q(z)$ we obtain $\big(u,E_{u,z}(y)\big)=(u,y)$, and using the equality $(u,x)=0$ as well we evaluate the composition $E_{u,x}\big(E_{u,z}(y)\big)$ as \[y+(z,y)u-(u,y)[z+q(z) \cdot u]+\big(x,y-(u,y)z\big)u-(u,y)[x+q(x) \cdot u].\] As the assertion about $u$ and $x+z$ is clear and the latter expression is easily seen to be $E_{u,x+z}(y)$, the fourth assertion follows, and with it the invertibility of $E_{u,x}$ (with inverse $E_{u,-x}$) completes the proof of the first assertion as well. Finally, Equation \eqref{bilform} and bilinearity show that $(u+v,x)=0$, and $q(u+v)=q(u)+q(v)$ annihilates $q(x)$ and $\sigma_{q}(x)$. Thus $E_{v,x}$ and $E_{u+v,x}$ are Euler transformations, and so is $E_{q(x)v,u}$ because $q(x)^{2}q(v)=0$. From the equalities $(u,v)=(u,x)=0$ we get $\big(u,E_{v,x}(y)\big)=(u,y)$, and together with $(v,x)=0$ the composition $E_{u,x}\big(E_{v,x}(y)\big)$ equals \[y+(x,y)v-(v,y)[x+q(x) \cdot v]+\big(x,y-(v,y)x\big)u-(u,y)[x+q(x) \cdot v],\] which amounts to $E_{u+v,x}(y)+q(x)(u,y)v-q(x)(v,y)u$. Since the pairings of $u$ and $v$ with themselves and with $x$ vanish, we get $\big(u,E_{u+v,x}(y)\big)=(u,y)$ and $\big(v,E_{u+v,x}(y)\big)=(v,y)$, and the vanishing of $q(u)q(x)(v,y)$ in Equation \eqref{Euler} for $E_{q(x)v,u}$ verifies the last asserted equality. This proves the proposition.
\end{proof}
Noting that while $\overline{q}$ was only defined on the finer quotient $\overline{M}$, the map $\sigma_{\overline{q}}$ is defined also on the coarser quotient $\overline{M}_{\sigma}$ from Lemma \ref{OMperp} (and is injective into $\overline{M}_{\sigma}^{*}$), we obtain the following consequence.
\begin{cor}
The map taking $u \in M_{q}^{\perp}$ and $x \in M$ to $E_{u,x}$ gives a group homomorphism from $M\otimes_{R}M_{q}^{\perp}$ to the kernel of the map from Lemma \ref{OMperp}. This map is the composition of the natural projection onto $\overline{M}_{\sigma}\otimes_{R}M_{q}^{\perp}$, the map $\sigma_{\overline{q}}\otimes\operatorname{Id}_{M_{q}^{\perp}}$ into $\overline{M}_{\sigma}\otimes_{R}M_{q}^{\perp}$, the natural map to $\operatorname{Hom}_{R}(\overline{M}_{\sigma},M_{q}^{\perp})$, and the identification from Lemma \ref{Odecom}. \label{xotimesu}
\end{cor}

\begin{proof}
First, if $u \in M_{q}^{\perp}$ then $E_{u,x}$ is defined for every $x \in M$, and its image in $\operatorname{O}(\overline{M},\overline{q})$ is trivial by the last assertion of Proposition \ref{propEuler} (since $\overline{u}=0$). The fact that it takes $y \in M$ to $y+(x,y)u$ (by equation \eqref{Euler} and the equality $(u,y)=0$) implies that it depends only on the image $\overline{x}$ of $x$ in $\overline{M}_{\sigma}$, and in particular $E_{q(x)v,u}=0$ for $v$ and $u$ in $M_{q}^{\perp}$. The formulae from Proposition \ref{propEuler} imply the bilinearity of $(x,u) \mapsto E_{u,x}$ on $M \times M_{q}^{\perp}$ (this is also visible from the formula for $E_{u,x}(y)$ for such $u$ and $x$), yielding the map from $M\otimes_{R}M^{\perp}$ into the desired kernel as well as its factorization through $\overline{M}_{\sigma}\otimes_{R}M^{\perp}$. The last assertion then follows from unwinding the definitions and recalling the translation by $\operatorname{Id}_{M}$ in Lemma \ref{Odecom}. This proves the corollary.
\end{proof}

\smallskip

Recall that the ideal $eR$ for $e\in\operatorname{Idem}(R)$ is a commutative ring with identity $e$, and then we have an orthogonal decomposition $M=eM\oplus(1-e)M$. Indeed, since $e(1-e)=0$ we can write each $x \in M$ uniquely as the sum of $ex \in eM$ and $(1-e)x\in(1-e)M$, and the $q$-values of the summands being in $eR$ and $(1-e)R$ respectively. We now make the following definition.
\begin{defn}
Let $x \in eM$ for some $e\in\operatorname{Idem}(R)$ such that $q(x)$, as an element of $eR$, lies in $(eR)^{\times}$. Then the map $r_{x}:M \to M$ defined by $r_{x}:y \mapsto y-\frac{(x,y)}{q(x)}x$ is called the \emph{$e$-reflection} associated with $x$. \label{eref}
\end{defn}
Note that for $x$ as in Definition \ref{eref} dividing by $q(x)$ means multiplying by its inverse in $eR$, but this is natural since the numerator $(x,y)$ is also in $eR$ (as $x \in eM$). We remark that while one may add to $x$ an element of $(1-e)M^{\perp}$, without altering the pairings $(x,y)$, the fact that this numerator is in $eR$ will not affect the quotient and will only take the unaffected part $ex$ of $x$, meaning that the $e$-reflection $r_{x}$ is the same map. We thus restrict attention to $x \in eR$. The classical reflections over rings, that are associated with $x \in M$ for which $q(x) \in R^{\times}$ as in \cite{[Ba]} and \cite{[Mc]}, are 1-reflections in the terminology of Definition \ref{eref} (i.e., the case $e=1$), and there are clearly no non-trivial 0-reflections.

Now, the fact that \[q\big(r_{x}(y)\big)=q\Big(y-\tfrac{(x,y)}{q(x)}x\Big)=q(y)-\tfrac{(x,y)}{q(x)}(x,y)+\tfrac{(x,y)^{2}}{q(x)^{2}}q(x)=q(y)\] shows that $r_{x}\in\operatorname{O}(M,q)$ for every element $x \in eM$ with $q(x)\in(eR)^{\times}$ for some $e\in\operatorname{Idem}(R)$, and since we have $r_{x}(x)=-x$ (because $(x,x)=2q(x)$), it follows that $r_{x}^{2}=\operatorname{Id}_{M}$. We define the \emph{reflection subgroup} $\operatorname{O}_{\mathrm{ref}}(M,q)$ to be the subgroup of $\operatorname{O}(M,q)$ generated by all $e$-reflections for all idempotents $e$, and, like with $\operatorname{O}(\overline{M},\overline{q})$, we shall use the notation $\operatorname{O}_{\mathrm{ref}}(\overline{M},\overline{q})$ also when $\overline{M}$ is not necessarily projective.
\begin{lem}
The group $\operatorname{O}_{\mathrm{ref}}(M,q)$ is contained in $\operatorname{O}_{M^{\perp}}(M,q)$. The subgroup $\operatorname{O}_{\mathrm{ref}}(\overline{M},\overline{q})$ of $\operatorname{O}(\overline{M},\overline{q})$ lies in the image of the map from Lemma \ref{OMperp}, and is the image of the subgroup $\operatorname{O}_{\mathrm{ref}}(M,q)$. \label{Oref}
\end{lem}

\begin{proof}
The first assertion follows immediately from the vanishing of $(x,y)$ wherever $y \in M^{\perp}$. Observe that if $x \in M$ satisfies the condition from Definition \ref{eref} for $e\in\operatorname{Idem}(R)$ and $\overline{x}$ is its image in $\overline{M}$ then $\overline{x}$ satisfies that condition for the same $e$ with $\overline{q}(\overline{x})=q(x)$, and thus $r_{x}\in\operatorname{O}_{\mathrm{ref}}(M,q)\subseteq\operatorname{O}_{M^{\perp}}(M,q)$ maps onto $r_{\overline{x}}$. The second assertion thus follows from the definition of $\operatorname{O}_{\mathrm{ref}}(M,q)$ and $\operatorname{O}_{\mathrm{ref}}(\overline{M},\overline{q})$. This proves the lemma.
\end{proof}

\begin{rmk}
Some submodule of the one from Corollary \ref{xotimesu} will already be contained in $\operatorname{O}_{\mathrm{ref}}(M,q)$. Indeed, if $x \in eM$ with $q(x)\in(eR)^{\times}$ for some element $e\in\operatorname{Idem}(R)$ as in Definition \ref{eref} and $u \in M^{\perp}$ with $eq(u)=0$ then $x-q(x)u$ also satisfies the condition from that definition for $e$ (indeed $q(x)u \in eM$ since $q(x) \in eR$), with $q\big(x-q(x)u\big)=q(x)$, and we have the $e$-reflections $r_{x}$ and $r_{x-q(x)u}$. The properties of $x$ and $u$ show that composing these reflections takes $y \in M$ to \[r_{x}\big[y-\tfrac{(x,y)}{q(x)}\big(x-q(x)u\big)\big]=y-\tfrac{(x,y)}{q(x)}x+\tfrac{(x,y)}{q(x)}\big(x+q(x)u\big)=y+(x,y)u\] (using $r_{x}(x)=-x$), and since $\sigma_{q}(u)=0$ and $q(u)\sigma_{q}(x)=0$ we find that the Euler transformation $E_{u,x}$ from Equation \eqref{Euler} is well-defined and takes $y$ to the same value. This holds, in particular, for every such $x$ with $u \in M_{q}^{\perp}$, so that if $M_{\mathrm{ref}}$ is the submodule of $M$ that is generated by such $x \in M$ then the image of $M_{\mathrm{ref}}\otimes_{R}M^{\perp}$ under the map from Corollary \ref{xotimesu} will be contained in $\operatorname{O}_{\mathrm{ref}}(M,q)$. This image can also be described in terms of the module $\overline{M}_{\mathrm{ref}}$, which is the $p$-image of $M_{\mathrm{ref}}$ in $\overline{M}$. However, in many cases $M_{\mathrm{ref}}$ is strictly contained in $M$: For example, if $R$ is an integral domain and $M$ is a rescaling of another quadratic module by an element that is neither a unit nor a zero-divisor then $M_{\mathrm{ref}}=\{0\}$ (and so is $\overline{M}_{\mathrm{ref}}$). Note that while the formula for $r_{x}$ yields an orthogonal transformation of $M$ wherever $q(x)$ divides $(x,y)$ for every $y \in M$ (this property is also invariant under rescalings), we do not treat such maps as $e$-reflection unless $q(x) \in R^{\times}$ (see also Remark \ref{liftrels} below). \label{Mref}
\end{rmk}

We say that $e$ and $\epsilon$ in $\operatorname{Idem}(R)$ are \emph{orthogonal idempotents} if $e\epsilon=0$, so that $e\oplus\epsilon=e+\epsilon$ in Lemma \ref{idem}. The second assertion in Lemma \ref{idem} is in terms of sums of orthogonal idempotents.
Here are some relations among the $e$-reflections from Definition \ref{eref}.
\begin{lem}
If $r_{x}$ is an $e$-reflection for some $e\in\operatorname{Idem}(R)$ and $\epsilon$ is another idempotent of $R$, then $\epsilon x$ satisfies the property from Definition \ref{eref} for $e\epsilon$, yielding the $e\epsilon$-reflection $r_{\epsilon x}$, which operates like $r_{x}$ on $\epsilon M$, and trivially on $(1-\epsilon)M$. For a third idempotent $\varepsilon$, the composition $r_{\epsilon}r_{\varepsilon}$ yields the $e(\epsilon\oplus\varepsilon)$-reflection $r_{(\epsilon\oplus\varepsilon)x}$, with the addition from Lemma \ref{idem}. In particular, if $x$ and $y$ in $M$ define an $e$-reflection $r_{x}$ and an $\epsilon$-reflection $r_{y}$ respectively, and $e$ and $\epsilon$ in $\operatorname{Idem}(R)$ are orthogonal, then $r_{x}r_{y}=r_{x+y}$. \label{rxrels}
\end{lem}

\begin{proof}
We have $\epsilon x \in e\epsilon M$, with $q(\epsilon x)=\epsilon q(x)\in(e\epsilon R)^{\times}$ (via the projection defined by multiplication by $\epsilon$), and as the expression $\frac{(\epsilon x,\epsilon y)}{q(\epsilon x)}\epsilon x$ from Definition \ref{eref} can clearly be written as $\epsilon\frac{(x,y)}{q(x)}x$, which is the same as $\frac{(x,y)}{q(x)}x$ for $y \in\epsilon M$ and vanishes if $y\in(1-\epsilon)M$, this proves the first assertion. For the second one, since $(1-\epsilon x)$ is perpendicular to $\epsilon x$ but $r_{\epsilon x}(\epsilon x)=-\epsilon x$ we get that \[r_{\epsilon x}\big(r_{\varepsilon x}(y)\big)=r_{\epsilon x}\big(y-\varepsilon\tfrac{(x,y)}{q(x)}x\big)=y-\epsilon\tfrac{(x,y)}{q(x)}x-(1-\epsilon)\varepsilon\tfrac{(x,y)}{q(x)}x+\epsilon\varepsilon\tfrac{(x,y)}{q(x)}x.\] Since the total coefficient of $-\frac{(x,y)}{q(x)}x$ here is $(1-\epsilon)\varepsilon+\epsilon(1-\varepsilon)$, which is the same as $\epsilon\oplus\varepsilon$ by the second assertion of Lemma \ref{idem}, our second assertion follows. For the third one, note that when $e$ and $\epsilon$ are orthogonal, the element $x+y\in(e+\epsilon)M$ satisfies $q(x+y)=q(x)+q(y)\in\big((e+\epsilon)R\big)^{\times}$ (since its images in the quotients $eR$ and $\epsilon R$ are $q(x)\in(eR)^{\times}$ and $q(y)\in(\epsilon R)^{\times}$ respectively), and we get the $(e+\epsilon)$-reflection $r_{x+y}$. Since $x=e(x+y)$ and $y=\epsilon(x+y)$, the third assertion is a consequence of the second one, because $e\oplus\epsilon=e+\epsilon$ and $(e+\epsilon)(x+y)=x+y$. This proves the lemma.
\end{proof}
The case $e=\epsilon=1$ (even over fields) shows that the last assertion in Lemma \ref{rxrels} does not hold without the orthogonality assumption on $e$ and $\epsilon$. Relations similar to those from that lemma appear in Remark \ref{liftrels} below.

\smallskip

Recall that for a non-degenerate quadratic module $(M,q)$, there is a definition of a degree in \cite{[Ba]} such that the determinant of an element of $\operatorname{O}(M,q)$ (defined via its action on $\bigwedge^{n}M$, where $n$ is the rank of $M$) is $-1$ raised to that degree (this determinant can be, in general, any element of $R$ that squares to 1, and if $R$ has idempotents then there can be many such possible determinants), and $\operatorname{SO}(M,q)$ is defined to be the kernel of that determinant. Here our modules need not be non-degenerate, and a general definition of a determinant might be very complicated. However, in several cases one can obtain a well-defined determinant in a natural way. This is the case for (typically non-orthogonal) transformations like those from the following lemma.
\begin{lem}
Given two vectors $z$ and $t$ in $M$, the map $T_{z,t}:M \to M$ taking $y \in M$ to $y+(t,y)z$ has determinant $1+(z,t)$. If there is an $R$-module homomorphism $\varphi$ from the ideal $(t,M):=\sigma_{q}(t)(M)=\{(t,x)|x \in M\}$ of $R$ into $R$ then by setting $T_{z,\varphi \circ t}$ to be the map $y \mapsto y+\varphi\big((t,y)\big)z$, its determinant is $1+\varphi\big((z,t)\big)$. \label{dettrans}
\end{lem}

\begin{proof}
We localize $M$ in some maximal ideal $\mathfrak{m}$ of $R$, and then $M_{\mathfrak{m}}$ is free, with some basis $\{x_{i}\}_{i=1}^{n}$. Denote by $v$ (resp. $u$) the (column) vector of the coordinates of $z$ (resp. $t$) in this basis, and by $G$ the resulting Gram matrix of the bilinear form arising from $(M_{\mathfrak{m}},q_{\mathfrak{m}})$ via Equation \eqref{bilform}. We can thus write the matrix describing the transformation $T_{t,z}$ as $I_{n}+vB$, where $I_{n}$ is the identity matrix of size $n$ and $B$ is the row vector $u^{t}G$, sending the vector of coordinates of $y$ to $(t,y)$. But the Weinstein--Aronszajn identity compares this determinant with that of $I_{1}+Bv=1+u^{t}Gv$, which gives the desired value because $u^{t}Gv=(z,t)$. Therefore $\det T_{z,t}$ is a number which equals $1+(z,t)$ at the localization of every maximal ideal $\mathfrak{m}$ of $R$, so that it is well-defined as $1+(z,t) \in R$. For the second assertion, the action of $\varphi$ will appear on the coefficients of $u^{t}G$ on both expressions, yielding the desired result. This proves the lemma.
\end{proof}

We can thus evaluate the determinant of Euler transformations.
\begin{cor}
For $u$ and $x$ in $M$ satisfying the equalities $(u,x)=0$ in $R$ and $q(u)\sigma_{q}(x)=q(u)q(x)\sigma_{q}(u)=0$ in $M^{*}$, the Euler transformation $E_{u,x}$ from Equation \eqref{Euler} has determinant 1. \label{detEuler}
\end{cor}

\begin{proof}
Since $u$ pairs with $(x,y)u$ to $2q(u)\sigma_{q}(x)(y)=0$, the Euler transformation $E_{u,x}$ is the composition $T_{x+q(x) \cdot u,-u} \circ T_{u,x}$ of two transformations of the form from Lemma \ref{dettrans}. The vanishing of $(u,x)$ and $q(x)(u,u)=q(u)\sigma_{q}(x)(x)$ shows that the determinants are $1+(u,x)=1$ and $1-\big(u,x+q(x) \cdot u\big)=1$ respectively, and as the determinant is clearly multiplicative when so defined, the composition $E_{u,x}$ indeed has the asserted determinant. This proves the corollary.
\end{proof}

For reflections, and more generally the $e$-refections from Definition \ref{eref}, the determinant also works well.
\begin{prop}
The determinant is a group homomorphism from $\operatorname{O}_{\mathrm{ref}}(M,q)$ into the image of $\big(\operatorname{Idem}(R),\oplus\big)$ inside $\mu_{2}(R) \subseteq R^{\times}$ by the homomorphism from Lemma \ref{idem}. \label{detref}
\end{prop}

\begin{proof}
Consider the $e$-reflection $r_{x}$ for an element $x \in eM$ with $q(ex)\in(eR)^{\times}$ as in Definition \ref{eref}. Then $(x,M) \subseteq eR$, in which $q(x)$ is invertible, and we have the $R$-linear map $\varphi:(x,M) \to R$ taking $(x,y)$ into $\frac{(x,y)}{q(x)} \in eR \subseteq R$. Our expression for $r_{x}$ is thus as $T_{-x,\varphi \circ x}$ in the notation of Lemma \ref{dettrans}. By this lemma it thus has a well-defined determinant $1-\varphi\big((x,x)\big)$, whose multiple by $1-e$ is $1-e$ since $(x,x) \in eR$ and so is its $\varphi$-image. On the other hand, after multiplying by $e$ we obtain $e-\frac{(x,x)}{q(x)}=e-2e=-e$ since on $eM$ the $\varphi$-image divides by $q(x)$, and $(x,x)=2q(x)$ by Remark \ref{kersigmaq}. Writing our determinant $d$ as $(1-e)d+ed=1-e-e=1-2e$, we deduce that every $e$-reflection, which also equals its inverse, has a well-defined determinant, which is the image $1-2e$ of $e$ in $\mu_{2}(R)$ under the asserted map. As once the determinant is well-defined on a set of invertible operators and their inverses it gives, by multiplicativity, a homomorphism from the group that they generate, we obtain the asserted homomorphism $\det:\operatorname{O}_{\mathrm{ref}}(M,q)\to\mu_{2}(R)$. This proves the proposition.
\end{proof}

We have seen in Remark \ref{Mref} that in some instances an Euler transformation $E_{u,x}$ is the composition of the two $e$-reflections $r_{x}$ and $r_{x-q(x)u}$, with the same $e$. Then Proposition \ref{detref} also gives the determinant $(1-2e)^{2}=1$ from Proposition \ref{detEuler}, though the latter holds for Euler transformations that are more general than those obtained in this way.

\begin{rmk}
The explicit elements of $\mathcal{O}_{M^{\perp}}(M,q)$ considered in Propositions \ref{detEuler} and \ref{detref} only give determinants in $\mu_{2}(R)$ that are in the image of the map from Lemma \ref{idem}. However, given $\delta$ that is not in that image (for example $\delta=X$ in $R=\mathbb{Z}[X]/(X^{2}-1)$), if there is a quadratic module $(M,q)$ that decomposes as an orthogonal direct summand $N \oplus L$ with $N$ of odd rank $m$ and $N \cap M^{\perp}=\{0\}$ (e.g., $M$ is non-degenerate of odd rank and $N=M$) then the map $\delta\operatorname{Id}_{N}\oplus\operatorname{Id}_{L}$ lies in $\mathcal{O}_{M^{\perp}}(M,q)$ and has a well-defined determinant $\delta^{m}=\delta$. \label{gendet}
\end{rmk}

We conclude by determining when an $e$-reflection is trivial. By Proposition \ref{detref} this can happen only when $eR$ is an $\mathbb{F}_{2}$-algebra (for the equality $1-2e=1$), i.e., when $e$ is in the kernel of the map from Lemma \ref{idem}. The precise condition is as follows.
\begin{prop}
The $e$-reflection $r_{x}$ from Definition \ref{eref} is trivial if and only if $x \in M^{\perp}$. \label{trivref}
\end{prop}

\begin{proof}
The fact that if $x \in M^{\perp}$ then $r_{x}=\operatorname{Id}_{M}$ is clear from the formula given in Definition \ref{eref}. Conversely, $r_{x}$ it trivial if and only if it is trivial locally, and an element of $M$ lies in $M^{\perp}$ if and only if its $M_{\mathfrak{m}}$-image lies in $M_{\mathfrak{m}}^{\perp}$ for every maximal ideal $\mathfrak{m}$ of $R$. Moreover, in any such localization the idempotent $e$ is taken to either 0 or 1, and since $x \in eR$ in Definition \ref{eref}, it vanishes in every $M_{\mathfrak{m}}$ for which $e=0$ in $R_{\mathfrak{m}}$. It thus suffices to consider the case where $e=1$ and $M$ if free, with basis $\{z_{i}\}_{i=1}^{n}$. We write $x=\sum_{i=1}^{n}a_{i}z_{i}$, which implies that the element $q(x) \in R^{\times}$, which equals $\sum_{i=1}^{n}a_{i}^{2}q(z_{i})+\sum_{i<j}a_{i}a_{j}(z_{i},z_{j})$, lies in the ideal generated by $\{a_{i}\}_{i=1}^{n}$. Now, the triviality of $r_{x}$ is equivalent to the condition that $(x,y)x=0$ for every $y \in M$, which means that $(x,y)a_{i}=0$ for every $i$, and thus $(x,y)=0$ by the triviality of the latter ideal. This proves the proposition.
\end{proof}
Proposition \ref{trivref} also shows that a trivial $e$-reflection $r_{x}$ can only exist when $eR$ is an $\mathbb{F}_{2}$-algebra, without applying Proposition \ref{detref}. Indeed, We have $2q(x)=(x,x)=0$ by Remark \ref{kersigmaq}, and inside $eR$ we can divide by $q(x)$ and obtain the vanishing of 2 there. As an example of trivial $2e$ but non-trivial $r_{x}$, we take $R=\mathbb{F}_{2}$, $e=1$, $M=\mathbb{F}_{2}x\oplus\mathbb{F}_{2}y$, $q$ the quadratic form taking $x$ to 1 and 0, $y$, and $x+y$ to 0 (with the natural pairing taking two distinct non-zero elements of $M$ to 1 and gives 0 otherwise). Then $r_{x}$ interchanges $y$ and $x+y$, and is thus non-trivial. Note that the maps interchanging $x$ with $x+y$ or $y$ are not reflections here, despite the $S_{3}$-symmetry of the pairing, because while $q(x)$ is in $\mathbb{F}_{2}^{\times}$, $q(y)$ and $q(x+y)$ are not.

\section{Clifford Groups and Paravectors \label{CGrp}}

As a ring, the Clifford algebra $\mathcal{C}:=\mathcal{C}(M,q)$ contains the group of units $\mathcal{C}^{\times}$, which is clearly preserved under the grade involution from Equation \eqref{invgrade} and under the transpose and Clifford involutions. However, recalling the embedding of $M$ into $\mathcal{C}$, a more important multiplicative group is the \emph{Clifford group}
\begin{equation}
\Gamma(M,q):=\{\alpha\in\mathcal{C}^{\times}|\alpha x\alpha'^{-1} \in M\mathrm{\ and\ }\alpha^{-1} x\alpha' \in M\ \forall x \in M\} \label{Clgrp}
\end{equation}
The twisted conjugation from Equation \eqref{Clgrp} produces an action of $\alpha\in\Gamma(M,q)$ on $M$, whose properties we describe in the following refinement of the statements from \cite{[Ba]}, \cite{[Mc]}, and others.
\begin{lem}
If $\alpha\in\mathcal{C}^{\times}$ satisfies $\alpha M\alpha'^{-1} \subseteq M$ then the map $x\mapsto\alpha x\alpha'^{-1}$ is an orthogonal transformation on $M$ that acts trivially on $M^{\perp}$. \label{Gammaorth}
\end{lem}

\begin{proof}
The fact that $q(y)=y^{2}=-yy'$ for $y \in M$ (via Lemma \ref{RMinv}) implies that \[q(\alpha x\alpha'^{-1})=-(\alpha x\alpha'^{-1})(\alpha x\alpha'^{-1})'=\alpha x\alpha'^{-1}\alpha'(-x)\alpha^{-1}=\alpha q(x)\alpha^{-1}=q(x)\] for every $x \in M$. Moreover, if $x \in M^{\perp}$ then $x\in\widetilde{Z}(\mathcal{C})$ by Remark \ref{twZCM}, meaning that $x\beta=\beta'x$ by the last assertion of Proposition \ref{ZCalg}, which by setting $\beta=\alpha'$ yields $\alpha x\alpha'^{-1}=x$ as desired. This proves the lemma.
\end{proof}

Note that the requirement in Equation \eqref{Clgrp}, which amounts to $\alpha M\alpha'=M$, seems stronger than requiring only $\alpha x\alpha'^{-1} \in M$, namely $\alpha M\alpha' \subseteq M$, as done in \cite{[Ba]}, \cite{[Mc]}, \cite{[Ab]}, and others. The following corollary clarifies this difference.
\begin{cor}
Every element satisfying the condition from Lemma \ref{Gammaorth} for which the transformation $x\mapsto\alpha x\alpha'^{-1}$ is surjective on $M$ is in $\Gamma(M,q)$. In particular, if $(M,q)$ satisfies the condition that every orthogonal transformation of $M$ that restricts to $\operatorname{Id}_{M^{\perp}}$ on $M^{\perp}$ is surjective then the Clifford group $\Gamma(M,q)$ from Equation \eqref{Clgrp} coincides with $\{\alpha\in\mathcal{C}^{\times}|\alpha x\alpha'^{-1} \in M\ \forall x \in M\}$. \label{invexist}
\end{cor}

\begin{proof}
Recalling that the kernel of an orthogonal transformation is contained in $M^{\perp}$, it follows that if such a transformation restricts to $\operatorname{Id}_{M^{\perp}}$ on $M^{\perp}$ then it is injective (for transformations arising from twisted conjugation as in Lemma \ref{Gammaorth} we also obtain the injectivity from such twisted conjugation by $\alpha\in\mathcal{C}^{\times}$ being invertible as a map on all of $\mathcal{C}$). Thus, under our assumptions, there is a bijective map $\varphi:M \to M$ (or, more precisely, an element $\varphi\in\operatorname{O}(M,q)$, which even lies in $\operatorname{O}_{M^{\perp}}(M,q)$) such that $\alpha x\alpha'^{-1}=\varphi(x)$ for every $x \in M$. Now, given $y \in M$, we know that $\beta:=\alpha^{-1} y\alpha'$ is the unique element $\mathcal{C}$ such that $\alpha\beta\alpha'^{-1}$. But $\varphi$ is invertible, and by setting $x:=\varphi^{-1}(y) \in M$ we obtain that $\alpha x\alpha'^{-1}=\varphi(x)=y$ in $\mathcal{C}$. The uniqueness of $\beta$ implies that $\beta=x \in M$, and thus $\alpha\in\Gamma(M,q)$ as desired. This proves the corollary.
\end{proof}
If Conjecture \ref{allMq} holds then so does the condition from Corollary \ref{invexist} for every quadratic module $(M,q)$, and our definition coincides with the one from previous references. However, without this our definition is the more precise one, and we have the following consequence of Lemma \ref{Gammaorth}.
\begin{cor}
There is a group homomorphism $\pi:\Gamma(M,q)\to\operatorname{O}_{M^{\perp}}(M,q)$. Its kernel is $\Gamma(M,q)\cap\widetilde{Z}(\mathcal{C})=\widetilde{Z}(\mathcal{C})^{\times}$, the groups of units of the algebra from Equation \eqref{twistZ}. \label{OCl}
\end{cor}

\begin{proof}
Lemma \ref{Gammaorth} applied to both $\alpha\in\Gamma(M,q)$ and its inverse shows the transformation $\pi(\alpha):x\mapsto\alpha x\alpha'^{-1}$ lies in $\operatorname{O}_{M^{\perp}}(M,q)$, and it is clear that the resulting map $\pi$ is a group homomorphism. An element $\alpha\in\Gamma(M,q)$ lies in $\ker\pi$ if and only if $\alpha x=x\alpha'$ for every $x \in M$, which means $\alpha\in\widetilde{Z}(\mathcal{C})$ by Equation \eqref{twistZ}. As we clearly have $\widetilde{Z}(\mathcal{C})\cap\mathcal{C}^{\times}=\widetilde{Z}(\mathcal{C})^{\times}$, the equality with $\Gamma(M,q)\cap\widetilde{Z}(\mathcal{C})$ is evident from Equation \eqref{Clgrp}. This proves the corollary.
\end{proof}

\smallskip

In Equation \eqref{Euler} and Definition \ref{eref} we described certain special elements of $\operatorname{O}_{M^{\perp}}(M,q)$. Using calculations similar to those from Subsection 4.1 of \cite{[Ba]}, these elements yield subgroup of $\operatorname{O}_{M^{\perp}}(M,q)$ that is contained in the image of the map $\pi$ from Corollary \ref{OCl}.
\begin{prop}
The image $\pi\big(\Gamma(M,q)\big)$ contains the subgroup generated by $\operatorname{O}_{\mathrm{ref}}(M,q)$ and by all the Euler transformations from Equation \eqref{Euler}. In particular, the image of $\pi\big(\Gamma(M,q)\big)$ in $\operatorname{O}(\overline{M},\overline{q})$ contains $\operatorname{O}_{\mathrm{ref}}(\overline{M},\overline{q})$ and all the Euler transformations in $(\overline{M},\overline{q})$, and its intersection with the kernel of the map from Lemma \ref{OMperp} contains the image of the map from Corollary \ref{xotimesu}. \label{Impi}
\end{prop}

\begin{proof}
Assume that $x \in M$ satisfies the condition from Definition \ref{eref} for some $e\in\operatorname{Idem}(R)$, and set $\alpha:=(1-e)+x$. The product $\alpha\alpha'$ is then $(1-e)-q(x) \in R$, with $q(x)\in(eR)^{\times}$, and multiplying by $(1-e)-\frac{e}{q(x)} \in R$ we obtain, using $e(1-e)=0$, the product $1-e+e=1$. It follows that $\alpha^{-1}=\big((1-e)-\frac{e}{q(x)}\big)\alpha'$, and thus $\alpha'^{-1}=\big((1-e)-\frac{e}{q(x)}\big)\alpha$. Since $1-e \in R$ is a central idempotent in $\mathcal{C}$ and $ex=x$ and $(1-e)x=0$ for $x \in eM$ we generalizing the argument from \cite{[Mc]} and others and obtain, for any $y \in M$, the expression \[\alpha y\alpha'^{-1}=\big((1-e)-\tfrac{e}{q(x)}\big)\big((1-e)+x\big)y\big((1-e)+x\big)=(1-e)y-\tfrac{xyx}{q(x)}.\] Applying Equation \eqref{xy+yx} and the defining relations for $\mathcal{C}$ from Equation \eqref{Cliffdef}, and recalling that $\frac{q(x)}{q(x)}$ is the unit $e$ of $eR$ and not 1, the latter expression is \[(1-e)y-\tfrac{(x,y)}{q(x)}x+\tfrac{xxy}{q(x)}=(1-e)y-\tfrac{(x,y)}{q(x)}x+ey=y-\tfrac{(x,y)}{q(x)}x=r_{x}(y) \in M.\] Since $r_{x}\in\operatorname{O}(M,q)$, Corollary \ref{invexist} implies that $\alpha\in\Gamma(M,q)$, and $\pi(\alpha)=r_{x}$ in the notation of Corollary \ref{OCl}. It follows that $\operatorname{O}_{\mathrm{ref}}(M,q)\subseteq\pi\big(\Gamma(M,q)\big)$, and $\operatorname{O}_{\mathrm{ref}}(\overline{M},\overline{q})$ is contained in the image of this group in $\operatorname{O}(\overline{M},\overline{q})$ via Lemmas \ref{OMperp} and \ref{Oref}.

Take now $u$ and $x$ in $M$ with $(u,x)=0$ and $q(u)q(x)=0$, and consider the element $\beta:=1-xu\in\mathcal{C}_{+}$. We have $(1-xu)(1-ux)=1-(x,u)+xq(u)x=1$, meaning that $\beta\in\mathcal{C}^{\times}$ with $\beta^{-1}=\beta'^{-1}=\beta^{*}=\overline{\beta}$. Moreover, we have $\beta=1+ux$ (since $(u,x)=0$), and similarly $\beta^{-1}=1+xu$. Now, for $y \in M$ we have \[\beta y\beta'^{-1}=(1-xu)y(1-ux)=y-(u,y)x+xyu+yxu+(u,y)xux-xyu^{2}x\] (using Equation \eqref{xy+yx} and the equality $(u,x)=0$), and if we also assume, after writing $u^{2}=q(u) \in R$, that $q(u)\sigma_{q}(x)=0$ and thus $q(u)(x,y)=0$, then after two other applications of Equation \eqref{xy+yx}, the latter expression becomes the right hand side of Equation \eqref{Euler}. As such transformations on $M$ were seen in Proposition \ref{propEuler} to be invertible, we deduce from Corollaries \ref{invexist} and \ref{OCl} that $\beta\in\Gamma(M,q)$ and $\pi(\beta)=E_{u,x}$. The first assertion is thus established via group generation, and the second and third ones are evident consequences of the first. This proves the proposition.
\end{proof}

\begin{rmk}
Take some $x \in eM$ with $q(x)\in(eR)^{\times}$, with $r_{x}$ as in Definition \ref{eref}. Then, for $\epsilon\in\operatorname{Idem}(R)$, the $e\epsilon$-reflection $r_{\epsilon x}$ from Lemma \ref{rxrels} arises from $(1-e\epsilon)+\epsilon x\in\Gamma(M,q)$. By taking another idempotent $\varepsilon$, the product of the latter element with the one associated with $r_{\varepsilon x}$ gives \[(1-e\epsilon)(1-e\varepsilon)+e\epsilon\varepsilon q(x)+[(1-\epsilon)\varepsilon+\epsilon(1-\varepsilon)]x.\] Since $q(x)\in(eR)^{\times}$, we can multiply this element by $(1-e\epsilon\varepsilon)+\frac{e\epsilon\varepsilon}{q(x)}$, which lies in $R^{\times}$, and this leaves the coefficient of $x$ in the previous expression, which is $(\epsilon\oplus\varepsilon)x$ by Lemma \ref{idem}, invariant. This shows that the product of the elements lying over $r_{\epsilon x}$ and over $r_{\varepsilon x}$ gives a multiplier from $R^{\times}\subseteq\widetilde{Z}(\mathcal{C})^{\times}$ times the one lying over $r_{(\epsilon\oplus\varepsilon)x}$, as Lemma \ref{rxrels} predicts. This includes the case $\epsilon=\varepsilon=1$ and the involutivity of $r_{x}$. For the case of $r_{x}$ and $r_{y}$ with orthogonal idempotents, the product of $(1-e)+x$ and $(1-\epsilon)+y$ gives $(1-e-\epsilon)+x+y$ directly. Similarly, if $u$, $x$, $z$, and $v$ satisfy the assumptions of Proposition \ref{propEuler} then the product of $1-ux$ and $1-vx$ gives $1-(u+v)x-q(x)uv$, and this coincides with the product of $1-q(x)uv$ and $1-(u+v)x$, yielding, via the proof of Proposition \ref{Impi}, the corresponding relation from Proposition \ref{propEuler}. The relation involving $a \in R$ there is trivial in terms of the proof of Proposition \ref{Impi}, but with $z$ the product of $1-ux$ and $1-uz$ differs from $1-u(x+z)$ in the multiplier $\gamma:=1-q(u)xz$, which can be non-trivial. However, the vectors $q(u)x$ and $z$ satisfy the conditions required for defining the Euler transformation $E_{q(u)x,z}$ in Equation \eqref{Euler}, and this transformation is trivial because $q(u)(x,y)$ and $q(u)(z,y)$ vanish for every $y \in M$. Thus $\gamma\in\widetilde{Z}(\mathcal{C})_{+}^{\times}=Z(\mathcal{C})_{+}^{\times}$, with inverse $\gamma^{*}=\overline{\gamma}$. Finally, in the situation described in Remark \ref{Mref}, the product of the inverse $\big((1-e)-\frac{e}{q(x)}\big)(1-e-x)$ of $(1-e)+x$ with $(1-e)+\big(x-q(x)u\big)$ equals $1-e+\frac{x}{q(x)}\big(x-q(x)u\big)=1-xu$ (since $x \in eR$ and $\frac{q(x)}{q(x)}=e$ in $eR$), which is indeed the element mapping to $E_{u,x}$ in Proposition \ref{Impi}. While the condition that $q(x)$ for $x \in eM$ divides $(x,y)$ for every $y$ suffices for the formula for the $e$-reflection $r_{x}$ to make sense and define an element of $\operatorname{O}(M,q)$ (as in the end of Remark \ref{Mref}), the element $\alpha=(1-e)+x$ lies in $\mathcal{C}^{\times}$, with inverse $\alpha'$ times an element of $R^{\times}$, only if $q(x)$ is a unit in $eR$. \label{liftrels}
\end{rmk}

From Proposition \ref{Impi} we obtain the following consequence.
\begin{cor}
If $\operatorname{O}_{\mathrm{ref}}(\overline{M},\overline{q})=\operatorname{O}(\overline{M},\overline{q})$ and $M\otimes_{R}M_{q}^{\perp}$ surjects, via Corollary \ref{xotimesu}, onto $\operatorname{Hom}_{R}(\overline{M}_{\sigma},M_{q}^{\perp})$, then the map $\pi:\Gamma(M,q)\to\operatorname{O}_{M^{\perp}}(M,q)$ from Corollary \ref{OCl} is surjective. This is always the case when $R$ is a field $\mathbb{F}$ of characteristic different from 2, a case in which $\operatorname{O}_{M^{\perp}}(M,q)=\operatorname{O}_{\mathrm{ref}}(M,q)$. \label{surj}
\end{cor}

\begin{proof}
The first assertion follows immediately from Proposition \ref{Impi}. For the second one, recall that when $R=\mathbb{F}$ the orthogonal group $\operatorname{O}(\overline{M},\overline{q})$ of the non-degenerate quadratic space $(\overline{M},\overline{q})$ is generated by reflections (this is the \emph{Cartan--Dieudonn\'{e} Theorem}---see, e.g., Corollary 4.3 of \cite{[MH]} for a simple proof). Moreover, the submodule $\overline{M}_{\mathrm{ref}}$ from Remark \ref{Mref} equals all of the non-degenerate space $\overline{M}$. Thus $\operatorname{O}_{\mathrm{ref}}(M,q)$ surjects onto $\operatorname{O}(\overline{M},\overline{q})$ in the map from Lemma \ref{OMperp} and contains the kernel of the projection, and thus it equals all of $\operatorname{O}_{M^{\perp}}(M,q)$. This proves the corollary.
\end{proof}
Recall that over $\mathbb{F}$, with $ch\mathbb{F}\neq2$, the non-degenerate space $\overline{M}$ (which is the same as $\overline{M}_{\sigma}$ by Remark \ref{kersigmaq}) is already unimodular (so the map from $\overline{M}$ to $\overline{M}^{*}$ is an isomorphism), and for vector spaces, the map from $\overline{M}^{*}\otimes_{\mathbb{F}}M^{\perp}$ to $\operatorname{Hom}_{\mathbb{F}}(\overline{M},M^{\perp})$ is also an isomorphism. This establishes the second condition in Corollary \ref{surj} regardless of Remark \ref{Mref}. However, over rings, even when $R$ is an integral domain in characteristic different from 2 and $\overline{M}$ is projective, the exact determination of the image of $\pi$ seems much more delicate. First, even when the map from $\overline{M}^{*}\otimes_{R}M^{\perp}$ to $\operatorname{Hom}_{R}(\overline{M},M^{\perp})$ is an isomorphism, the non-degenerate quadratic module $(\overline{M},\overline{q})$ need not be unimodular, and thus the map from Corollary \ref{xotimesu} need not surject onto the kernel from Lemma \ref{OMperp}. Moreover, in many such modules there are no reflections at all (see the case mentioned at the end of Remark \ref{Mref}). While the appearance, in Proposition \ref{Impi}, of Euler transformations other than those from Corollary \ref{xotimesu} may help to increase the known image in some cases, a general argument for the surjectivity of $\pi$ may be more difficult. We leave these questions for further research, but generalize Corollary \ref{surj} in a special case in Theorem \ref{SESGamma} below.

\smallskip

As in Lemma 2.2 of \cite{[Mc]}, and generalizing Proposition 3.2.1 of \cite{[Ba]}, one proves the following result, in which we recall that given $\alpha\in\mathcal{C}:=\mathcal{C}(M,q)$, the element $\operatorname{N}(\alpha):=\alpha\overline{\alpha}$ is called the \emph{norm} of $\alpha$.
\begin{prop}
If $\alpha\in\mathcal{C}$ is in $\Gamma(M,q)$ then so are $\alpha'$ from Equation \eqref{invgrade} as well as the images $\alpha^{*}$ and $\overline{\alpha}$ of $\alpha$ under the other two involutions. We also have $\pi(\alpha')=\pi(\alpha)\in\operatorname{O}_{M^{\perp}}(M,q)$ as well as $\pi(\alpha^{*})=\pi(\overline{\alpha})=\pi(\alpha)^{-1}$ with $\pi$ from Corollary \ref{OCl}. In particular, for every $\alpha\in\Gamma(M,q)$ we have $N(\alpha)\in\widetilde{Z}(\mathcal{C})^{\times}$. \label{Gammainv}
\end{prop}

\begin{proof}
Recall that if $x \in M$ then $x^{*}=x$ and $x'=\overline{x}=-x$, and consider the elements $\alpha x\alpha'^{-1}$ and $\alpha^{-1}x\alpha'$ of $M$. By setting $\varphi=\pi(\alpha)\in\operatorname{O}_{M^{\perp}}(M,q)$, we can write these elements as $\varphi(x)$ and $\varphi^{-1}(x)$ respectively. Applying the grade involution, the Clifford involution, and the transposition to the resulting equalities yields \[-\alpha'x\alpha^{-1}=-\varphi(x) \in M\qquad\mathrm{\ and\ }\qquad-\alpha'^{-1}x\alpha=-\varphi^{-1}(x) \in M\] for the grade involution from Equation \eqref{invgrade}, \[-\alpha^{-*}x\overline{\alpha}=-\varphi(x) \in M\qquad\mathrm{\ and\ }\qquad-\alpha^{*}x\overline{\alpha}^{-1}=-\varphi^{-1}(x) \in M\] using the Clifford involution (where $\alpha^{-*}$ is a shorthand for $(\alpha^{*})^{-1}=(\alpha^{-1})^{*}$), and the transposition yields \[\overline{\alpha}^{-1}x\alpha^{*}=\varphi(x) \in M\qquad\mathrm{\ and\ }\qquad\overline{\alpha}x\alpha^{-*}=\varphi^{-1}(x) \in M.\] This proves that $\alpha'$, $\alpha^{*}$, and $\overline{\alpha}$ are in $\Gamma(M,q)$ (note the inversion in the roles of the equalities for $\alpha$ and for $\alpha^{-1}$ in the order-inverting involutions), and determines their $\pi$-images to be the asserted ones. It follows that $N(\alpha)$ lies in $\Gamma(M,q)$ and has a trivial $\pi$-image, yielding the last assertion via Corollary \ref{OCl}. This proves the proposition.
\end{proof}

Altogether we obtain the following result.
\begin{thm}
Let $R$ be a ring, and assume that $(M,q)$ is a quadratic module over $R$, with Clifford algebra $\mathcal{C}:=\mathcal{C}(M,q)$, such that the reflection group $\operatorname{O}_{\mathrm{ref}}(M,q)$ and the set of Euler transformations of $M$, as defined in Equation \eqref{Euler}, generate the subgroup $\operatorname{O}_{M^{\perp}}(M,q)$ of the elements of $\operatorname{O}(M,q)$ that reduce to $\operatorname{Id}_{M^{\perp}}$ on the submodule $M^{\perp}$ from Equation \eqref{kernel}. Let $S$ be another ring containing $R$, and assume that $(M_{S},q_{S})$ is generated freely by the orthogonal set $\{x_{i}\}_{i=1}^{n}$, and that no element of $\{2\}\cup\{q_{S}(x_{i})|1 \leq i \leq n\}\setminus\{0\}$ is a zero-divisor in $S$. Then the Clifford group $\Gamma(M,q)$ from Equation \eqref{Clgrp} sits in the short exact sequence \[1\to\textstyle{R^{\times}\oplus\bigoplus_{r>0}\bigwedge^{r}M^{\perp}}\to\Gamma(M,q)\stackrel{\pi}{\to}\operatorname{O}_{M^{\perp}}(M,q)\to1.\] This sequence is operated on by the three involutions, where the action of the one from Equation \eqref{invgrade} is trivial on $\operatorname{O}_{M^{\perp}}(M,q)$, while the other two act as inversion on the latter group. \label{SESGamma}
\end{thm}

\begin{proof}
Proposition \ref{Impi} and the generation assumption show that the map $\pi$ is the one from Corollary \ref{OCl} is surjective, with kernel $\widetilde{Z}(\mathcal{C})^{\times}$. Theorem \ref{ZtildeZ} determines the latter, under our assumptions, as the group of invertible elements in $\bigwedge^{*}M^{\perp}=R\oplus\bigoplus_{r>0}\bigwedge^{r}M^{\perp}$, and as we have seen before Proposition \ref{rad}, the direct sum over $r>0$ consists only of nilpotent elements, of degree at most $n$. Since the projection from $\bigwedge^{*}M^{\perp}$ onto the component $R$ is a ring homomorphism, invertible elements must be in the asserted set $R^{\times}\oplus\bigoplus_{r>0}\bigwedge^{r}M^{\perp}$, which consists of multiples of elements from $R^{\times}$ and element of the form $1-\nu$ for nilpotent $\nu$. Since both elements are invertible (the inverse of the latter is $\sum_{j=0}^{n-1}\nu^{j}$, with $\nu^{0}=1$), the asserted set is indeed the desired group $\widetilde{Z}(\mathcal{C})^{\times}$. The assertion concerning the involution follows immediately from Proposition \ref{Gammainv}. This proves the theorem.
\end{proof}

As Corollary \ref{surj} shows, the short exact sequence from Theorem \ref{SESGamma} exists for every quadratic space over a field of characteristic different from 2. As an non-trivial example of a case where these results apply and $R$ is not a field, consider $R=\mathbb{Z}$, so that quadratic modules are even lattices, and take $M$ to be a \emph{reflection lattice}, i.e., a lattice whose orthogonal group is generated by reflections (the standard even unimodular positive definite lattice $E_{8}$ of rank 8 has this property). For non-degenerate such lattices, Theorem \ref{SESGamma} expresses $\Gamma(M,q)$ as a double cover of $\operatorname{O}(M,q)$. Let, however, $(M,Q)$ be a quadratic module admitting an $e$-reflection for some $e$ for which $eR$ is not an $\mathbb{F}_{2}$-algebra, and set $q$ to be the rescaling of $Q$ by an element whose image in $\epsilon R$ is neither a unit nor a zero-divisor for any $\epsilon\in\operatorname{Idem}(R)$. Then the $e$-reflection of $(M,Q)$ also gives an element of $\operatorname{O}_{M^{\perp}}(M,q)$ (with the same $M^{\perp}$), with determinant $1-2e$ via Proposition \ref{detref}, but $(M,q)$ contains no elements yielding any $\epsilon$-reflections, so that $\operatorname{O}_{\mathrm{ref}}(M,q)$ is trivial and the group generated by Euler transformations only contains element of determinant 1. In this case we cannot obtain surjectivity via the proof of Theorem \ref{SESGamma}, which exemplifies the complexity of the investigation of the image of $\pi$ in general. We remark, however, that if elements of $\operatorname{O}_{M^{\perp}}(M,q)$ of the sort described in Remark \ref{gendet} exist then there does not seem to be a way for constructing pre-images for them in $\Gamma(M,q)$, or even finding $\pi$-images with such determinants.

\smallskip

We have a special decomposition of $\Gamma(M,q)$ in case $\widetilde{Z}(\mathcal{C})=R$ (so that in particular $(M,q)$ is non-degenerate, since we have already seen that $M^{\perp}\subseteq\widetilde{Z}(\mathcal{C})$ in general). For an idempotent $e$ of $R$, we say that $\alpha\in\mathcal{C}$ is \emph{homogenous of degree $e$} if $\alpha\in(1-e)\mathcal{C}_{+} \oplus e\mathcal{C}_{-}$, or equivalently if $(1-e)\alpha\in\mathcal{C}_{+}$ and $e\alpha\in\mathcal{C}_{-}$. For $e=0$ and $e=1$ we get the elements of $\mathcal{C}_{+}$ and $\mathcal{C}_{-}$ respectively, indeed classically called homogenous of degree 0 and 1 respectively. Corollary \ref{OCl}, Proposition \ref{Impi}, and the proof of Theorem \ref{SESGamma} have the following consequence.
\begin{cor}
Assume that $\widetilde{Z}(\mathcal{C})=R$, and that the group $\operatorname{O}_{M^{\perp}}(M,q)=\operatorname{O}(M,q)$ (recall that $M^{\perp}$ is trivial) is generated by $e$-reflections and Euler transformations. Denote by $\operatorname{Idem}(M,q)$ the subgroup of the group $(\operatorname{Idem}(R),\oplus)$ from Lemma \ref{idem} that is generated by all those $e\in\operatorname{Idem}(R)$ for which $(M,q)$ admits an $e$-reflection via definition \ref{eref}. Then $\Gamma(M,q)$ decomposes as $\biguplus_{e\in\operatorname{Idem}(M,q)}\Gamma(M,q)_{e}$, such that $\Gamma(M,q)_{0}$ is $\Gamma(M,q)_{+}=\Gamma(M,q)\cap\mathcal{C}_{+}$, a normal subgroup of $\Gamma(M,q)$, and every $\Gamma(M,q)_{e}$ is a coset of this normal subgroup. $\pi$-images of elements in $\Gamma(M,q)_{e}$ have determinant $1-2e$ in $\operatorname{O}(M,q)$. \label{Gammadecom}
\end{cor}

\begin{proof}
Since the kernel of the map $\pi$ from Corollary \ref{OCl} is just $\widetilde{Z}(\mathcal{C})^{\times}=R^{\times}$, elements of $\Gamma(M,q)$ are determined by their $\pi$-images only up to invertible scalars. Now, Proposition \ref{Impi} produces lifts of the generators of $\operatorname{O}(M,q)$, namely Euler transformations (of determinant 1) and $e$-reflections (with determinant $1-2e$). Moreover, these the lifts of the Euler transformation in that proposition lie in $\Gamma(M,q)\cap\mathcal{C}_{+}$ (and is thus homogenous of degree 0), while the lift of an $e$-reflection lies in $(1-e)\mathcal{C}_{+} \oplus e\mathcal{C}_{-}$, i.e., it is homogenous of degree $e$. Since multiplication by $R^{\times}$ does not affect these properties, this is true for every possible lift of each generator. Now, it is clear, e.g., via Lemma \ref{idem}, that if $\alpha$ is homogenous of degree $e$ and $\beta$ is homogenous of degree $\epsilon$ for another $\epsilon\in\operatorname{Idem}(R)$ then $\alpha\beta$ is homogenous of degree $e\oplus\epsilon$. Since each of the generators of $\Gamma(M,q)$ (namely, the lifts from Proposition \ref{Impi} and $R^{\times}$) is homogenous of some degree, we thus have a group homomorphism from $\Gamma(M,q)$ into $(\operatorname{Idem}(R),\oplus)$, whose kernel is $\Gamma(M,q)_{+}$ and whose image is $\operatorname{Idem}(M,q)$ by definition. This gives the decomposition in question, and as $\det\circ\pi$ is the composition of the latter group homomorphism with the map into $\mu_{2}(R)$ from Lemma \ref{idem} (since both maps coincide of the generators), the second assertion follows as well. This proves the corollary.
\end{proof}
Corollary \ref{Gammadecom} extends the classical decomposition of the Clifford group of a non-degenerate quadratic form over a field of characteristic different from 2 into the homogenous parts in the grading from Equation \eqref{grading}, with the even part mapping to the special orthogonal group while the odd part yields orthogonal maps of determinant $-1$. Note that without the non-degeneracy assumption these results are no longer valid, because elements of $\ker\pi$, as seen, for example, in Theorem \ref{SESGamma}, need not be homogenous in general.

\smallskip

Following \cite{[Ma]} and \cite{[Mc]}, we call elements of $R \oplus M\subseteq\mathcal{C}:=\mathcal{C}(M,q)$ \emph{paravectors}. In correspondence with the equalities $xx=xx^{*}=x^{*}x=q(x)$ for $x \in M$, or more precisely $xx'=x'x=x\overline{x}=\overline{x}x=-q(x)$ (by Lemma \ref{RMinv}), we note that if $\xi \in R \oplus M$ then $\xi\xi'=\xi'\xi=\xi\overline{\xi}=\overline{\xi}\xi=-q_{R}(\xi)$, where $q_{R}$ is the quadratic form on $M_{R}:=R \oplus M$ taking $\xi=a+x \in M_{R}$, with $a \in R$ and $x \in M$, to $q(x)-a^{2}$. If we denote the bilinear form associated with $q_{R}$ via Equation \eqref{bilform} by $(\cdot,\cdot)_{R}$, then for our $\xi$ and for $\eta=b+y \in M_{R}$ we obtain $(\xi,\eta)_{R}=2ab+(x,y)$.

It follows that the corresponding submodules from Equation \eqref{kernel} are
\begin{equation}
M_{R}^{\perp}=R[2] \oplus M^{\perp} \qquad\mathrm{and}\qquad M_{R,q_{R}}^{\perp}=\{a+x \in M_{R}^{\perp}|a^{2}+q(x)=0\} \supseteq M_{q}^{\perp}, \label{kerR}
\end{equation}
Therefore the quotient $\overline{M}_{R,\sigma}$ analogous to the one from Lemma \ref{OMperp} is just $R/2R\oplus\overline{M}_{\sigma}$, we have \[\overline{M}_{R}=(R\oplus\overline{M})\big/\big\{a+x \in R[2]\oplus\overline{M}^{\perp}\big|a^{2}=q(x)\big\},\] and $\overline{M}_{R}^{\perp}$ is $R[2]\oplus\overline{M}^{\perp}$ divided by the same submodule.

We also have the analogue of Equation \eqref{xy+yx} (modified via Lemma \ref{RMinv}), giving
\begin{equation}
\xi\eta'+\eta\xi'=\xi'\eta+\eta'\xi=-(\xi,\eta)_{R}\cdot1\mathrm{\ in\ }\mathcal{C}\mathrm{\ for\ every\ }\xi\mathrm{\ and\ }\eta\mathrm{\ in\ }R \oplus M\subseteq\mathcal{C}. \label{pairR}
\end{equation}
The module $M_{R}$ contains the element $1 \in R$, with $q_{R}(1)=-1 \in R^{\times}$, yielding the reflection $r_{1}$ on $M_{R}$. Note that in contrast to equalities $x=-x'=-\overline{x}$ for every $x \in M$ from Lemma \ref{RMinv}, the element $-\xi'=-\overline{\xi}$ no longer equals $\xi$ for every $\xi \in M_{R}$, but rather $r_{1}(\xi)$ (the equality $\xi^{*}=\xi$ does hold on all of $M_{R}$---Lemma \ref{RMinv} again). The orthogonality of this reflection is also visible in the defining equality of $q_{R}(\xi)$ as both $\xi\xi'$ and $\xi'\xi$, as well as in Equation \eqref{pairR}. Thus Lemma \ref{rxrels} has the following consequence.
\begin{cor}
Every idempotent $e\in\operatorname{Idem}(R)$ defines, as an element of $M_{R}$, an $e$-reflection $r_{e}$, which is explicitly defined by $r_{e}:\eta\mapsto(1-e)\eta-e\eta'$ for $\eta \in M_{R}$. This $e$-reflection is trivial if and only if $e \in R[2]$. The composition of the $\epsilon$-reflection $r_{\epsilon}$ for another element $\epsilon\in\operatorname{Idem}(R)$ gives the reflection $r_{e\oplus\epsilon}$, with the addition from Lemma \ref{idem}. \label{reMR}
\end{cor}

\begin{proof}
The existence of $r_{e}$, the formula for $r_{e}(\eta)$, and the product rule are the special case $x=1 \in M_{R}$ of Lemma \ref{rxrels}. As
the pairing $-2be$ of $e$ and $\eta=b+y \in M_{R}$ vanishes identically if and only if $e \in R[2]$, the remaining assertion follows from Proposition \ref{trivref}. This proves the corollary.
\end{proof}

\smallskip

In addition to the Clifford group, another subgroup of the group of units $\mathcal{C}^{\times}$ of $\mathcal{C}=\mathcal{C}(M,q)$ is the \emph{paravector Clifford group}
\begin{equation}
\widetilde{\Gamma}(M,q):=\{\alpha\in\mathcal{C}^{\times}|\alpha\xi\alpha'^{-1} \in R \oplus M\mathrm{\ and\ }\alpha^{-1}\xi\alpha' \in R \oplus M\ \forall\xi \in R \oplus M\}. \label{pvCl}
\end{equation}
We obtain the following analogue of Lemma \ref{Gammaorth} and Corollary \ref{OCl}.
\begin{prop}
Given $\alpha\in\widetilde{\Gamma}(M,q)$, the map $\widetilde{\pi}(\alpha):R \oplus M \to R \oplus M$ defined by $\xi\mapsto\alpha\xi\alpha'^{-1}$ lies in $\operatorname{O}_{M_{R}^{\perp}}(M_{R},q_{R})$. This produces a group homomorphism $\widetilde{\pi}:\widetilde{\Gamma}(M,q)\to\operatorname{O}_{M_{R}^{\perp}}(M_{R},q_{R})$, whose kernel consists of those $\alpha\in\widetilde{Z}(\mathcal{C})^{\times}$ whose $\mathcal{C}_{-}$-part $\alpha_{-}$ from Equation \eqref{grading} satisfies $2\alpha_{-}=0$. In particular, if 2 is not a zero-divisor in $R$ then $\ker\widetilde{\pi}=\widetilde{Z}(\mathcal{C})_{+}^{\times}=Z(\mathcal{C})_{+}^{\times}$. \label{paraorth}
\end{prop}

\begin{proof}
Recall the equality $\eta\eta'=-q_{R}(\eta)$ for $\eta \in M_{R}$, and that $x \in M^{\perp}$ satisfies $x\beta=\beta'x$ for every $\beta\in\mathcal{C}$. But for $a$ in the second part $R[2]$ of $M_{R}^{\perp}$ from Equation \ref{kerR} we have $2a\beta_{-}=0$ and thus $a\beta=\beta'a$, implying that the first assertion can be proved just like Lemma \ref{Gammaorth}. The homomorphism property is clear, and since an element of $\ker\widetilde{\pi}$ must act trivially on $M$, it lies in $\widetilde{Z}(\mathcal{C})^{\times}$ by Corollary \ref{OCl}. However, elements of $\ker\widetilde{\pi}$ act trivially also on 1, meaning that $\alpha\alpha'^{-1}=1$, and therefore $\alpha=\alpha'$. Lemma \ref{annby2} thus completes the determination of the kernel, and the last equality is given in Proposition \ref{ZCalg}. This proves the proposition.
\end{proof}
Note that elements $\alpha=\alpha_{+}+\alpha_{-}$ of $\widetilde{Z}(\mathcal{C})$ (or equivalently of $Z(\mathcal{C})$) for which $\alpha_{-}\in\mathcal{C}_{-}[2]$, as in Proposition \ref{paraorth}, are in the intersection $Z(\mathcal{C})\cap\widetilde{Z}(\mathcal{C})$. However, this intersection can sometimes contain additional elements: For example, when $M_{q}^{\perp}$ is free, of odd rank $d$, and $R$ is not an $\mathbb{F}_{2}$-algebra, a generator of $\bigwedge^{d}M_{q}^{\perp}$ lies in $Z(\mathcal{C})\cap\widetilde{Z}(\mathcal{C})$ and in $\mathcal{C}_{-}$ but not in $\mathcal{C}_{-}[2]$.

The proof of Corollary \ref{invexist} shows that if some element $\alpha\in\mathcal{C}^{\times}$ satisfies $\alpha(R \oplus M)\alpha'^{-1} \subseteq R \oplus M$, and the orthogonal map arising from $\alpha$ is surjective (and thus in $\operatorname{O}_{M_{R}^{\perp}}(M_{R},q_{R})$), then $\alpha$ is in $\widetilde{\Gamma}(M,q)$. Similarly, the condition $\alpha(R \oplus M)\alpha'^{-1} \subseteq R \oplus M$ determines the elements of $\widetilde{\Gamma}(M,q)$ in case the quadratic module $(M_{R},q_{R})$ satisfies the condition from that corollary (in particular, if Conjecture \ref{allMq} holds). Also, in the characteristic 2 case from Remark \ref{char2}, the kernel of $\widetilde{\pi}$ is all of $\widetilde{Z}(\mathcal{C})^{\times}=Z(\mathcal{C})^{\times}$, and if we also have $\sigma_{q}=0$ (i.e., when $M^{\perp}=M$) then this kernel is all of $\mathcal{C}^{\times}$, also when $M_{q}^{\perp} \neq M$ as in Remark \ref{kersigmaq}.

\smallskip

On the quadratic space $(M_{R},q_{R})$ one defines the Euler transformations $E_{\upsilon,\xi}$ for $\upsilon$ and $\xi$ in $M_{R}$ with $(\upsilon,\xi)_{R}=0$, $q_{R}(\upsilon)q_{R}(\xi)=0$, and $q_{R}(\upsilon)(\xi,\eta)_{R}=0$ for every $\eta \in M_{R}$ via Equation \eqref{Euler}, as well as the $e$-reflection $r_{\xi}$ for $\xi \in eM_{R}$ with $q_{R}(\xi)\in(eR)^{\times}$ for some $e\in\operatorname{Idem}(R)$. Note that if $\xi=x \in M$, and if $\upsilon=u \in M$ in the first map, then since $(x,1)_{R}=0$ and $(u,1)_{R}=0$ the Euler transformation $E_{u,x}$ and the $e$-reflection $r_{x}$ on $M_{R}$ are just $\operatorname{Id}_{R} \oplus E_{u,x}$ and $\operatorname{Id}_{R} \oplus r_{x}$ with $E_{u,x}$ or $r_{x}$ on $M$, making the resulting slight abuse of notation allowable.

Recalling that Proposition \ref{detref} gives a well-defined determinant homomorphism from the reflection group of any quadratic module, we denote its kernel, as expected, using the notation $\operatorname{SO}_{\mathrm{ref}}$. In case our quadratic module is $(M_{R},q_{R})$, we have additional properties of $\det$ and its kernel.
\begin{lem}
The image of $\det$ on $\operatorname{O}_{\mathrm{ref}}(M_{R},q_{R})$ coincides with that of the homomorphism from Lemma \ref{idem}. The kernel $\operatorname{SO}_{\mathrm{ref}}(M_{R},q_{R})$ is generated by the elements $r_{\xi} \circ r_{e}$ for $e\in\operatorname{Idem}(R)$ and $\xi \in eM_{R}$ with $q_{R}(\xi)\in(eR)^{\times}$, where $r_{e}$ is the $e$-reflection from Corollary \ref{reMR}. \label{rxir1}
\end{lem}

\begin{proof}
The existence of $r_{e}$ from Corollary \ref{reMR} for every $e\in\operatorname{Idem}(R)$ yields the first assertion via Proposition \ref{detref} and its proof. Next, on one hand, every such product $r_{\xi}r_{e}$ lies in $\operatorname{SO}_{\mathrm{ref}}(M_{R},q_{R})$ by the same proposition. On the other hand, a general element of $\operatorname{O}_{\mathrm{ref}}(M_{R},q_{R})$ is a product $r_{\xi_{1}} \ldots r_{\xi_{k}}$ with $\xi_{i}$ an $e_{i}$-reflection, and it lies in $\operatorname{SO}_{\mathrm{ref}}(M_{R},q_{R})$ if and only if $\epsilon:=\bigoplus_{i=1}^{k}e_{i}$ lies in the kernel of the map from Lemma \ref{idem}. But this means that $\epsilon \in R[2]$, so that we can multiplying our original element by $r_{\epsilon}$ from the right without changing it (as it is trivial by Corollary \ref{reMR}), and Lemma \ref{rxrels} allows us to express this trivial $\epsilon$-reflection as the product of $r_{e_{i}}$ for $1 \leq i \leq k$. Now, if $\xi \in eM_{R}$ satisfies $q_{R}(\xi)\in(eR)^{\times}$ (so that it defines the $e$-reflection $r_{\xi}$ via Definition \ref{eref}) and $\varphi\in\operatorname{O}(M_{R},q_{R})$, then $\varphi(\xi)$ satisfies the same condition for the same idempotent $e$ and we have $r_{\varphi(\xi)}=\varphi r_{\eta}\varphi^{-1}$. Using this relation we set $\tilde{\xi}_{i}$ to be $r_{\epsilon_{i}}(\xi_{i})$ for $\epsilon_{i}:=\bigoplus_{j=1}^{i-1}e_{j}$, and our orthogonal transformation becomes the product of $r_{\tilde{\xi}_{i}} \circ r_{e_{i}}$, which are of the desired form. This proves the lemma.
\end{proof}

\smallskip

We can now prove the analogue of Proposition \ref{Impi} for the paravector Clifford group from Equation \eqref{pvCl}.
\begin{prop}
The image of the map $\widetilde{\pi}$ from Proposition \ref{paraorth} contains all the Euler transformations $E_{\upsilon,\xi}$, as defined in Equation \eqref{Euler}, for $\upsilon$ and $\xi$ in $R \oplus M$ with $(\upsilon,\xi)=0$ and $q_{R}(\upsilon)$ annihilating $q_{R}(\xi)$ in $M_{R}$ and $\sigma_{q_{R}}(\xi)$ in $M_{R}^{*}$, as well as the subgroup $\operatorname{SO}_{\mathrm{ref}}(M_{R},q_{R})$ from Lemma \ref{rxir1}. \label{tildepi}
\end{prop}
Since $\overline{M}_{R}$ need not be projective, our method for defining determinants may not be valid for $\operatorname{O}_{\mathrm{ref}}(\overline{M}_{R},\overline{q}_{R})$, and thus we do not consider a group like $\operatorname{SO}_{\mathrm{ref}}(\overline{M}_{R},\overline{q}_{R})$ in Proposition \ref{tildepi}. The analogue of the last statement of Proposition \ref{Impi}, involving the kernel of the map from Lemma \ref{OMperp} and the image of that from Corollary \ref{xotimesu} (both associated with $(M_{R},q_{R})$ now), clearly holds here as well.

\begin{proof}
We follow the proof of Proposition \ref{Impi}, but since care must be taken regarding the action of the grade involution from Equation \eqref{invgrade}, we give the details. Note that a small modification of Equation \eqref{pairR} yields the equality $\xi\eta+\eta'\xi'=(\xi,\eta')\cdot1$ for $\xi$ and $\eta$ in $R \oplus M\subseteq\mathcal{C}$.

Take, for some idempotent $e$, an element $\xi \in e(R \oplus M)$ with $q_{R}(\xi)\in(eR)^{\times}$. Then $\alpha:=(1-e)+\xi$ has the inverse $(1-e)-\frac{\xi'}{q_{R}(\xi)}$, so that $\alpha'^{-1}:=(1-e)-\frac{\xi}{q_{R}(\xi)}$. Once again in writing $\eta \in R \oplus M$ as $(1-e)\eta+e\eta$, the relation $e(1-e)=0$ and the centrality of $1-e$ expresses $\alpha\eta\alpha'^{-1}$ as $(1-e)\eta-\frac{e}{q_{R}(\xi)}\xi\eta\xi$. Our modification of Equation \eqref{pairR} and the fact that $\xi \in e(R \oplus M)$ transform the latter expression into \[(1-e)\eta-\tfrac{e[(\xi,\eta')_{R}\xi+\eta'\xi'\xi]}{q_{R}(\xi)}=(1-e)\eta-e\eta'+\tfrac{(\xi,\eta')}{q_{R}(\xi)}\xi=r_{\xi}\big((1-e)\eta-e\eta'\big) \in R \oplus M.\] As Lemma \ref{reMR} shows that this element is the image of $\eta$ under the element $r_{\xi} \circ r_{e}$ of $\operatorname{O}(M_{R},q_{R})$, we deduce that $\alpha\in\widetilde{\Gamma}(M,q)$ with $\widetilde{\pi}(\alpha)=r_{\xi} \circ r_{e}$, and thus $\operatorname{SO}_{\mathrm{ref}}(M_{R},q_{R})$ is contained in $\widetilde{\pi}\big(\widetilde{\Gamma}(M,q)\big)$ by Lemma \ref{rxir1}.

Considering now such $\upsilon$ and $\xi$, we set $\beta:=1+\xi\upsilon'$ (this element need no longer be in $\mathcal{C}_{+}$, and the opposite sign compared to the proof of Proposition \ref{Impi} will compensate for the minus sign in Equation \eqref{pairR}, not appearing in Equation \eqref{xy+yx}). As $(1+\xi\upsilon')(1+\upsilon\xi')=1-(\xi,\upsilon)+\xi q_{R}(\upsilon)\xi=1$ (via Equation \eqref{pairR}), we get that $\beta\in\mathcal{C}^{\times}$ with inverse $\overline{\beta}$, and $\beta'^{-1}=\beta^{*}$. Once again $\beta$ equals $1-\upsilon\xi'$ as well, and thus $\beta^{-1}=1-\xi'\upsilon$. Given $\eta \in R \oplus M$, we obtain, via Equation \eqref{pairR} and the equality $(\upsilon,\xi)=0$, the expression \[\beta\eta\beta'^{-1}=(1+\xi\upsilon')\eta(1+\upsilon'\xi)=\eta-(\upsilon,\eta)_{R}\xi-\xi\eta'\upsilon-\eta\xi'\upsilon-(\upsilon,\eta)_{R}\xi\upsilon'\xi-\xi\eta'\upsilon\upsilon'\xi.\] Expressing $\upsilon\upsilon'$ as $-q_{R}(\upsilon)$, using Equation \eqref{pairR} again twice, and substituting the values $(\upsilon,\xi)=q_{R}(\upsilon)q_{R}(\xi)=q_{R}(\upsilon)(\xi,\eta)=0$ and $\xi'\xi=-q_{R}(\xi)$, the latter expression becomes $\eta+(\xi,\eta)_{R}\upsilon-(\upsilon,\eta)_{R}[\xi+q_{R}(\xi)\upsilon]$. As this is indeed $E_{\upsilon,\xi}(\eta)$ by Equation \eqref{Euler}, we obtain that $\beta\in\widetilde{\Gamma}(M,q)$ and that $\widetilde{\pi}(\beta)=E_{\upsilon,\xi}$ as desired. This proves the proposition.
\end{proof}

We have seen in Lemma \ref{reMR} that if $eR$ is an $\mathbb{F}_{2}$-algebra then $r_{e}$ is trivial, and in particular it is contained in the image of $\widetilde{\pi}$. On the other hand, we have the following generalization of part 4 of Proposition 3.6 of \cite{[EGM]}.
\begin{prop}
Assume that the ring $eR$ is contained in a ring $S$, and that the extension $\big((eM)_{S},(eq)_{S})$ of the quadratic module $(eM,eq)$ over $eR$ admits a free orthogonal generating set $\{x_{i}\}_{i=1}^{n}$. Assume further that neither 2 nor any non-zero value $(eq)_{S}(x_{i})$ are zero-divisors in $S$ (so that in particular 2 is not a zero-divisor in $eR$). Then the reflection $r_{e}$ is not contained in the image of the map $\widetilde{\pi}$ from Proposition \ref{paraorth}. \label{notsurj}
\end{prop}

\begin{proof}
Assume that there exists $\alpha\in\widetilde{\Gamma}(M,q)$ with $\widetilde{\pi}(\alpha)=r_{e}$. This means that $\alpha$ is in $\mathcal{C}^{\times}$ and satisfies $\alpha\xi=\big((1-e)\xi-e\xi'\big)\alpha'$ for every $\xi \in R \oplus M$. Since for $x \in M$ this becomes $\alpha x=x\alpha'$, such $\alpha$ must lie in $\widetilde{Z}(\mathcal{C})$ from Equation \eqref{twistZ} by definition (hence in $\widetilde{Z}(\mathcal{C})^{\times}$), and by taking $\xi=1-e$ and $\xi=e$ we get $(1-e)\alpha=(1-e)\alpha'$ and $e\alpha=-e\alpha'$. This implies that $(1-e)\alpha\in\mathcal{C}_{+}\oplus\mathcal{C}_{-}[2]$ in the decomposition from Equation \eqref{grading} like in Proposition \ref{paraorth} (indeed, the action of $(1-e)\alpha$ on $(1-e)M_{R}$ is trivial), and $e\alpha$ is in $\mathcal{C}_{+}[2]\oplus\mathcal{C}_{-}$, or equivalently in $e\mathcal{C}_{+}[2] \oplus e\mathcal{C}_{-}$. Moreover, recalling that 2 is not a zero-divisor in $eR$ and that $e\mathcal{C}_{+}$ is a flat $eR$-module, we deduce that $e\alpha_{+}=0$ and $e\alpha\in\mathcal{C}_{-}$.

Now, since $\alpha\in\widetilde{Z}(\mathcal{C})^{\times}$, we deduce that $e\alpha$ is in $e\widetilde{Z}(\mathcal{C})^{\times}$, which is contained in the group $\widetilde{Z}(e\mathcal{C})^{\times}$ of units of the twisted center of the Clifford algebra $e\mathcal{C}=\mathcal{C}(eM,eq)$ for the quadratic module $(eM,eq)$ over $eR$. But under our assumptions, Theorem \ref{ZtildeZ} determines the latter twisted center as $\bigwedge^{*}(eM)^{\perp}$, and we saw before Proposition \ref{rad} that this is the direct sum of $eR$ and nilpotent elements. It follows that $\widetilde{Z}(e\mathcal{C})_{-}$ consists only of nilpotent elements, and thus its intersection with $\widetilde{Z}(e\mathcal{C})^{\times}$, in which $e\alpha$ was seen to lie, is empty. Therefore an element $\alpha\in\widetilde{\Gamma}(M,q)$ such that $\widetilde{\pi}(\alpha)=r_{e}$ cannot exist, as desired. This proves the proposition.
\end{proof}
We expect that the conditions in Proposition \ref{notsurj} can be relaxed. In fact, we make the following conjecture.
\begin{conj}
Let $A$ be a commutative ring with unit in which $2\neq0$. Take a quadratic module $(N,Q)$ over $A$, and set $\mathcal{C}:=\mathcal{C}(N,Q)$. Then the intersection of $\widetilde{Z}(e\mathcal{C})^{\times}$ with $\mathcal{C}_{+}[2]\oplus\mathcal{C}_{-}$ is empty, i.e., no element $\alpha\in\widetilde{Z}(e\mathcal{C})^{\times}$ satisfies $2\alpha_{+}=0$. \label{emptyint}
\end{conj}
As the heart of the proof of Proposition \ref{notsurj} is that $e\alpha$ must lie in the intersection from Conjecture \ref{emptyint} (with $A=eR$, $N=eM$, and $Q=eq$), we deduce that if this conjecture is true then Proposition \ref{notsurj} holds wherever $eR$ is not an $\mathbb{F}_{2}$-algebra. Indeed, the proof of that proposition shows that it is valid under any assumption yielding the emptiness of that intersection.

\smallskip

We also have the analogous statement to Proposition \ref{Gammainv} for $\widetilde{\Gamma}(M,q)$.
\begin{prop}
Given $\alpha\in\mathcal{C}$, the elements $\alpha$, $\alpha'$, $\alpha^{*}$, and $\overline{\alpha}$ are in $\widetilde{\Gamma}(M,q)$ together. Assuming this is the case, then $\widetilde{\pi}(\alpha')$ is the conjugate of $\widetilde{\pi}(\alpha)$ by $r_{1}$ in $\operatorname{O}_{M_{R}^{\perp}}(M_{R},q_{R})$, $\widetilde{\pi}(\overline{\alpha})$ is the inverse of $\widetilde{\pi}(\alpha)$, and $\widetilde{\pi}(\alpha^{*})$ is the $r_{1}$-conjugate of that inverse. In addition, if $\alpha\in\widetilde{\Gamma}(M,q)$ then $N(\alpha)$ lies in $\widetilde{Z}(\mathcal{C})^{\times}$ and $2N(\alpha)_{-}$ vanishes in $\widetilde{Z}(\mathcal{C})$. \label{invpara}
\end{prop}

\begin{proof}
For $\xi \in R \oplus M$ we have $\alpha\xi\alpha'^{-1} \in R \oplus M$ and $\alpha^{-1}\xi\alpha' \in R \oplus M$ by definition, and set $\varphi=\widetilde{\pi}(\alpha)\in\operatorname{O}_{M_{R}^{\perp}}(M_{R},q_{R})$. Then the latter elements are $\varphi(\xi)$ and $\varphi^{-1}(\xi)$ respectively, and we get similar equalities with $\xi'=\overline{\xi}$ (we have $\xi^{*}=\xi$ for $\xi \in R \oplus M$). We apply the grade involution from Equation \eqref{grading} and the Clifford involution to the equality with $\xi'$ and the transposition to the one with $\xi$ itself and obtain \[\alpha'\xi\alpha^{-1}=\varphi(\xi')' \in R \oplus M\qquad\mathrm{\ and\ }\qquad\alpha'^{-1}\xi\alpha=\varphi^{-1}(\xi')' \in R \oplus M\] involving the grade involution from Equation \eqref{invgrade}, \[\alpha^{-*}\xi\overline{\alpha}=\varphi(\xi')' \in R \oplus M\qquad\mathrm{\ and\ }\qquad\alpha^{*}\xi\overline{\alpha}^{-1}=\varphi^{-1}(\xi')' \in R \oplus M\] for the Clifford involution (again using $\alpha^{-*}$ for $(\alpha^{*})^{-1}=(\alpha^{-1})^{*}$), and with the transposition this becomes \[\overline{\alpha}^{-1}\xi\alpha^{-*}=\varphi(\xi) \in R \oplus M\qquad\mathrm{\ and\ }\qquad\overline{\alpha}\xi\alpha^{-*}=\varphi^{-1}(\xi) \in R \oplus M.\] This shows that $\alpha'$, $\alpha^{*}$, and $\overline{\alpha}$ are in $\widetilde{\Gamma}(M,q)$ when $\alpha$ is (with the same role inversion), and since $\xi\mapsto-\xi'$ is the action of $r_{1}$ by Corollary \ref{reMR} (and $-\operatorname{Id}_{M_{R}}$ is central, of course), the relations between the $\widetilde{\pi}$-images of these elements are also established. Once again this implies that the norm $N(\alpha)=\alpha\overline{\alpha}$ of $\alpha\in\widetilde{\Gamma}(M,q)$ is also in $\widetilde{\Gamma}(M,q)$ and with a trivial $\widetilde{\pi}$-image, so that the remaining assertion follows from Proposition \ref{paraorth}. This proves the proposition.
\end{proof}

\smallskip

The analogue of Theorem \ref{SESGamma} for the paravector Clifford group is as follows.
\begin{thm}
Consider a ring $R$, a quadratic module $(M,q)$ over $R$, its Clifford algebra $\mathcal{C}:=\mathcal{C}(M,q)$, and the extended quadratic module $(M_{R},q_{R})$ with $M_{R}=R \oplus M$ and $q_{R}(a+x)=q(x)-a^{2}$ as above. Assume that the group $\operatorname{O}_{M_{R}^{\perp}}(M_{R},q_{R})$, consisting of those elements of $\operatorname{O}(M_{R},q_{R})$ that act trivially on $M_{R}^{\perp}$ from Equation \eqref{kerR}, is generated by the group $\operatorname{O}_{\mathrm{ref}}(M_{R},q_{R})$ of $e$-reflections on $(M_{R},q_{R})$ together with the Euler transformations of $M_{R}$ defined as in Equation \eqref{Euler}. In addition, assume that the scalar extension $(M_{S},q_{S})$ of $(M,q)$ to some ring $S$ containing $R$, in which 2 is not a zero-divisor, admits an orthogonal set $\{x_{i}\}_{i=1}^{n}$ of free generators, and that every non-zero value $q_{S}(x_{i})$ is also not a zero-divisor in $S$. Then there is a well-defined group homomorphism $\det$ from $\operatorname{O}_{M_{R}^{\perp}}(M_{R},q_{R})$ into the group $\mu_{2}(R)$ of square-roots of 1 in $R$, whose kernel we denote by $\operatorname{SO}_{M_{R}^{\perp}}(M_{R},q_{R})$, and the paravector Clifford group $\widetilde{\Gamma}(M,q)$ from Equation \eqref{pvCl} appears in the short exact sequence \[1\to\textstyle{R^{\times}\oplus\bigoplus_{s>0}\bigwedge^{2s}M^{\perp}}\to\widetilde{\Gamma}(M,q)\stackrel{\widetilde{\pi}}{\to}\operatorname{SO}_{M_{R}^{\perp}}(M_{R},q_{R})\to1.\] The three involutions operate on this sequence, with the action of the one from Equation \eqref{invgrade} on $\operatorname{SO}_{M_{R}^{\perp}}(M_{R},q_{R})$ being via conjugation by $r_{1}\in\operatorname{O}_{M_{R}^{\perp}}(M_{R},q_{R})$, the Clifford involution inverts elements of this group, and the transposition inverts and conjugates by $r_{1}$. \label{SESpara}
\end{thm}

\begin{proof}
The fact that $\det$ is defined on $e$-reflections and Euler transformations implies, via the generation assumption, that it is defined on all of $\operatorname{O}_{M_{R}^{\perp}}(M_{R},q_{R})$. Moreover, since the proof of Lemma \ref{rxir1} presents $\operatorname{O}_{\mathrm{ref}}(M_{R},q_{R})$ as the semi-direct product of $\operatorname{SO}_{\mathrm{ref}}(M_{R},q_{R})$ and the group $\{r_{e}|e\in\operatorname{Idem}(R)\}$ from Corollary \ref{reMR}, we get a similar decomposition for $\operatorname{O}_{M_{R}^{\perp}}(M_{R},q_{R})$, and it follows that $\operatorname{SO}_{M_{R}^{\perp}}(M_{R},q_{R})$ is generated by Euler transformations and $\operatorname{SO}_{\mathrm{ref}}(M_{R},q_{R})$. Since the conditions of Proposition \ref{notsurj} are satisfied for every $0 \neq e\in\operatorname{Idem}(R)$, the image of the map $\widetilde{\pi}$ from Proposition \ref{paraorth} is precisely $\operatorname{SO}_{M_{R}^{\perp}}(M_{R},q_{R})$. By this proposition and the fact that 2 is not a zero-divisor on the flat module $\mathcal{C}_{-}$, the kernel of $\widetilde{\pi}$ is $\widetilde{Z}(\mathcal{C})_{+}^{\times}$, and as in the proof of Theorem \ref{SESGamma}, Theorem \ref{ZtildeZ} and nilpotence consideration compares this kernel with the asserted group. Finally, Proposition \ref{invpara} provides the assertion about the involutions. This proves the theorem.
\end{proof}

\begin{rmk}
Note that the kernel of the projections in the short exact sequences from Theorems \ref{SESGamma} and \ref{SESpara}, in which the norms of elements of $\Gamma(M,q)$ and $\widetilde{\Gamma}(M,q)$ lie by Propositions \ref{Gammainv} and \ref{invpara}, can be larger than $R^{\times}$. However, in the theory of Vahlen groups, initiated in \cite{[V]} and extended in \cite{[EGM]} and then more generally in \cite{[Mc]}, the isomorphisms of these groups is not to the full Clifford group (or the paravector one), but rather to the subgroup consisting of those elements with norm in $R^{\times}$ (this was seen in \cite{[EGM]} and in \cite{[Mc]} to indeed be a subgroup). Note that all the elements from Propositions \ref{Impi} and \ref{tildepi} satisfy this property, so that under the assumptions from Theorems \ref{SESGamma} and \ref{SESpara}, restricting to each of these subgroups yields the same surjective map. The kernel, however, will be smaller, and will consists of those elements in $\textstyle{R^{\times}\oplus\bigoplus_{s>0}\bigwedge^{2s}M^{\perp}}$ or $\textstyle{R^{\times}\oplus\bigoplus_{r>0}\bigwedge^{r}M^{\perp}}$ that have norm in $R^{\times}$. \label{Vahlen}
\end{rmk}
It will be interesting to see if the assumptions in Theorems \ref{SESGamma} and \ref{SESpara} can be relaxed such that their results, as well as the modifications from Remark \ref{Vahlen}, still hold.

\smallskip

Recalling the homogenous elements of degree $e\in\operatorname{Idem}(R)$ considered in Corollary \ref{Gammadecom}.
\begin{prop}
Assume that $\alpha\in\mathcal{C}$ is homogenous of degree $e$ and invertible. Then it lies in $\Gamma(M,q)$ if and only if it is in $\widetilde{\Gamma}(M,q)$. In this situation the orthogonal map $\widetilde{\pi}(\alpha)\in\operatorname{O}_{M_{R}^{\perp}}(M_{R},q_{R})$ decomposes the direct sum of $(1-2e)\operatorname{Id}_{R}$ on $R$ and $\pi(\alpha)$ on $M$. \label{ehom}
\end{prop}

\begin{proof}
By writing $\alpha$ as $\alpha_{+}+\alpha_{-}$ as in Equation \eqref{grading}, homogeneity of degree $e$ and the fact that $1-2e\in\mu_{2}(R)$ multiplies $1-e$ by 1 and $e$ by $-1$ imply that $\alpha'=(1-2e)\alpha$. Moreover, by combining the homogeneity with invertibility, we deduce that $\alpha_{+}$ is invertible in $(1-e)\mathcal{C}$ and $\alpha_{-}$ is invertible in $e\mathcal{C}$. Therefore $\alpha^{-1}$ is also homogenous of degree $e$, and so is $\alpha'^{-1}=(1-2e)\alpha^{-1}$. Consider now $\alpha\beta\alpha'^{-1}$ for some $\beta\in\mathcal{C}$. Multiplying it by $(1-e)$ gives $(1-e)\alpha_{+}\beta\alpha_{+}^{-1}$ its $e$-multiple is $-e\alpha_{-}\beta\alpha_{-}^{-1}$ (the inverses are in $(1-e)\mathcal{C}$ and $e\mathcal{C}$ respectively, as Clifford algebras themselves), and both conjugations preserve the grading from Equation \eqref{grading}. Therefore this operation preserves $R \oplus M$ (with $R\subseteq\mathcal{C}_{+}$ and $M\subseteq\mathcal{C}_{-}$) if and only if it preserves $M$ and $R$. Recalling that $\alpha'=(1-2e)\alpha$, we deduce that $R$ itself is always preserved (and the action on it is the asserted one), and the fact that the same holds for the action of $\alpha^{-1}$ combines with the definitions of the groups in Equation \eqref{Clgrp} and \eqref{pvCl} to yield the first assertion. The second one immediately follows via Corollary \ref{OCl}, Proposition \ref{paraorth} and the action that we determined on $R$. This proves the proposition.
\end{proof}

\begin{rmk} 
Let $\alpha\in\widetilde{\Gamma}(M,q)$ be homogenous elements of degree $e$. Another property of $\alpha$, which immediately follows from Proposition \ref{ehom}, is that $\tilde{\pi}(\alpha)$ commutes with $r_{1}$ in $\operatorname{O}_{M_{R}^{\perp}}(M_{R},q_{R})$, and thus is coincides with $\tilde{\pi}(\alpha')$ by Proposition \ref{invpara} (indeed, the ratio between $\alpha'$ and $\alpha$ is in $\mu_{2}(R) \subseteq R^{\times}$). Altogether, for an element $\alpha\in\widetilde{\Gamma}(M,q)$ one can consider five conditions: (1) $\alpha$ is homogenous of degree $e$; (2) $\widetilde{\pi}(\alpha)$ preserves $R$ inside $M_{R}$; (3) $\widetilde{\pi}(\alpha)$ preserves $M$ inside $M_{R}$; (4) $\alpha$ lies also in $\Gamma(M,q)$; And (5) $\widetilde{\pi}(\alpha)$ commutes with $r_{1}$. Proposition \ref{ehom} shows that Condition (1) implies all the others, and Condition (4) is equivalent to Condition (3) for $\alpha$ and for $\alpha^{-1}$ (these conditions will be equivalent in case Conjecture \ref{allMq} is true). As for the failing of the other implications in general, note that Condition (2), still with the multiplier $1-2e$, will hold wherever $(1-e)\alpha\in\mathcal{C}_{+}\oplus\mathcal{C}_{-}[2]$ and $e\alpha\in\mathcal{C}_{+}[2]\oplus\mathcal{C}_{-}$, and such elements, also in $\widetilde{\Gamma}(M,q)$ are not necessarily homogenous (in fact, while the degree of homogeneity in Condition (1) is uniquely determined by $\alpha$, the multiplier $1-2e$ here remains the same if we change $e$ by an element of the kernel of the map from Lemma \ref{idem}, and this will not affect the expressions with the 2-torsion). Moreover, the fact that $R$ is perpendicular to $M$ does not prove that Condition (2) implies Condition (3), as in principal elements of $M$ can be taken by $\widetilde{\pi}(\alpha)$ to the larger submodule $R[2] \oplus M$ (Condition (2) does imply Condition (5) though). To see that Condition (3) implies neither Condition (2) nor Condition (5), consider the element $\alpha=\xi=1+x \in R \oplus M$, where $x \in M^{\perp}_{q}$. Its $\widetilde{\pi}$-image is $r_{\xi}$ by Proposition \ref{tildepi}, which operates on $M$ as $\operatorname{Id}_{M}$, but takes 1 to $1+2x$, and its conjugate by $r_{1}$ is the reflection in $1-x$, sending 1 to $1-2x$. Thus Condition (3), as well as (4), is satisfied for $\alpha$, but if $4x\neq0$ in $M$ then both Conditions (2) and (5) fail for $\alpha$. Finally, Condition (5) is equivalent to $\widetilde{\pi}(\alpha)$ sending $M$ into $R[2] \oplus M$ and taking $R$ into $R \oplus M[2]$. To see that it does not imply any of the other conditions, assume that $M$ contains an element $x$ with $4x=0$, $2x\neq0$, $2x \in M^{\perp}$, but $x \not\in M^{\perp}$ (it is not hard to construct a quadratic module $(M,q)$ over some ring $R$ containing such a vector). Take again $\alpha=\xi=1+x\in(R \oplus M)\cap\widetilde{\Gamma}(M,q)$, with $\widetilde{\pi}(\alpha)=r_{\xi}$ once again, which again takes 1 to $1+2x$ but sends $y \in M$ to $y-(x,y)(1+x)$. Our assumptions on $x$ mean that our element $\alpha$ satisfies Condition (5) but not any of the other conditions. \label{condint}
\end{rmk}
All of the explicit elements appearing in Remark \ref{condint} lie, in fact, inside the subgroup of $\widetilde{\Gamma}(M,q)$ mentioned in Remark \ref{Vahlen} (they actually all have norm 1). Thus Proposition \ref{ehom} and Remark \ref{condint} are equally valid for this group. As another hypothetical situation in which Condition (2) in Remark \ref{condint} is unrelated to Condition (1) there, consider the case where the multiplier of the action on 1 (namely the scalar $\delta \in R^{\times}$ such that $\alpha'=\delta\alpha$), which must be in $\mu_{2}(R)$ by preserving $q_{R}$, does not lie in the image of the map from Lemma \ref{idem}. It is an interesting question whether such elements of $\widetilde{\Gamma}(M,q)$ can exist.


\medskip

\noindent\textsc{Einstein Institute of Mathematics, the Hebrew University of Jerusalem, Edmund Safra Campus, Jerusalem 91904, Israel}

\noindent E-mail address: zemels@math.huji.ac.il

\end{document}